\documentclass[12pt, onesided]{amsart}

\usepackage{amsmath}
\usepackage{amssymb}
\usepackage{amsthm}
\usepackage[all,cmtip]{xy}
\usepackage{color}

\theoremstyle{plain}
\newtheorem{thm}{Theorem}[section]
\newtheorem{prop}[thm]{Proposition}
\newtheorem{lem}[thm]{Lemma}
\newtheorem{cor}[thm]{Corollary}

\newtheorem{fact}[thm]{Fact}

\theoremstyle{definition}
\newtheorem{defn}[thm]{Definition}
\newtheorem{eg}[thm]{Example}

\theoremstyle{remark}
\newtheorem{rem}[thm]{Remark}
\newtheorem{claim}[thm]{Claim}

\DeclareMathOperator{\Cl}{Cl}
\DeclareMathOperator{\codim}{codim}
\DeclareMathOperator{\Def}{Def}

\DeclareMathOperator{\Exc }{Exc}
\DeclareMathOperator{\Ext}{Ext}

\DeclareMathOperator{\Sing}{Sing}
\DeclareMathOperator{\Spec}{Spec}
\DeclareMathOperator{\depth}{depth}

\DeclareMathOperator{\Ker}{Ker}
\DeclareMathOperator{\Image}{Im}
\DeclareMathOperator{\Gal}{Gal}

\DeclareMathOperator{\Gr}{Gr}

\DeclareMathOperator{\sm}{sm}
\DeclareMathOperator{\Hom}{Hom}

\DeclareMathOperator{\Art}{Art}

\begin{document}

\title
[Deformations of weak $\mathbb{Q}$-Fano $3$-folds]
{Deformations of weak $\mathbb{Q}$-Fano $3$-folds}
\subjclass[2010]{Primary  14B07, 14J30, 14J45; Secondary 14B15}
\keywords{Deformation theory, weak $\mathbb{Q}$-Fano 3-folds, $\mathbb{Q}$-smoothing}

\author{Taro Sano}
\address{Department of Mathematics, Faculty of Science, Kobe University, 1-1, Rokkodai, Nada-ku,  Kobe, 657-8501, Japan}
\email{tarosano@math.kobe-u.ac.jp}
\thanks{The author was partially supported by Max Planck Institute for Mathematics, JSPS fellowship for Young Scientists,  JSPS KAKENHI
Grant Numbers JP15J03158, JP16K17573, JP24224001, and JST tenure track program. }

\maketitle

\begin{abstract} 
	We prove that a weak $\mathbb{Q}$-Fano $3$-fold with terminal singularities has unobstructed deformations. 
	By using this result and computing some invariants of a terminal singularity, 
	we provide two results on global deformation of a weak $\mathbb{Q}$-Fano 
	$3$-fold.   
 We also treat a stacky proof of the unobstructedness of deformations of a $\mathbb{Q}$-Fano 
 $3$-fold.
\end{abstract}

\tableofcontents

\section{Introduction} 

In this paper, we consider algebraic varieties over the complex number field $\mathbb{C}$. 

\begin{defn}
Let $X$ be a normal projective $3$-fold. We say that $X$ is a {\it weak $\mathbb{Q}$-Fano $3$-fold} 
(resp. {\it $\mathbb{Q}$-Fano $3$-fold}) 
if $X$ has only terminal singularities and $-K_X$ is a nef and big (resp. ample) divisor. 
\end{defn}

Weak $\mathbb{Q}$-Fano $3$-folds naturally appear in the birational study of a $\mathbb{Q}$-Fano $3$-fold. 
(cf. \cite{MR2177195}) 
In this paper, we first study the deformation of a weak $\mathbb{Q}$-Fano $3$-fold. 

The following is a main result of this paper. 

\begin{thm}\label{thm:wqfano3intro}
	Deformations of a weak $\mathbb{Q}$-Fano $3$-fold are unobstructed.  
	\end{thm} 
	
	The author proved the unobstructedness for a $\mathbb{Q}$-Fano $3$-fold (\cite[Theorem 1.7]{MR3419958}). 
	Minagawa proved it for a weak Fano $3$-fold with only terminal Gorenstein singularities(\cite[Main Theorem (1)]{MR1860839}). 
	Theorem \ref{thm:wqfano3intro} is a generalization of these results. 
	
	By using Theorem \ref{thm:wqfano3intro}, we study the existence of a $\mathbb{Q}$-smoothing 
of a weak $\mathbb{Q}$-Fano $3$-fold. 
Recall that a $3$-fold terminal singularity has a {\it $\mathbb{Q}$-smoothing}, that is, 
a deformation to a $3$-fold with only quotient singularities. 
In general, a weak $\mathbb{Q}$-Fano $3$-fold does not have a $\mathbb{Q}$-smoothing 
(cf. \cite[Example 3.7]{MR1860839}). 
	We give a partial result in this direction as follows. 
	
	\begin{thm}\label{thm:wqfanoQsmintro}
	 Let $X$ be a weak $\mathbb{Q}$-Fano $3$-fold. 
	\begin{enumerate}
	\item[(i)] Then $X$ can be deformed to a weak $\mathbb{Q}$-Fano $3$-fold $X_t$ with the following property: for all 
	$p_t \in \Sing X_t$, the invariant ``$\mu^{(1)}(X_t, p_t)$'' vanishes. 
	\item[(ii)] Assume that $X$ has a ``global index one cover'' which is $\mathbb{Q}$-factorial. 
	Then $X$ can be deformed to a $3$-fold with only quotient singularities and $A_{1,2}/4$-singularities. 
	\end{enumerate}
	\end{thm}	

The invariant $\mu^{(1)}$ of singularities appeared in Theorem \ref{thm:wqfanoQsmintro}(i) is an analogue 
of the invariant appeared in \cite[Section 2]{MR1358982} (See Definition \ref{defn:DBinvterminalb11}). 
An {\it $A_{1,2}/4$-singularity} in (ii) is the terminal singularity $(x^2+y^2+z^3 +u^2 =0)/ \mathbb{Z}_4(1,3,2,1)$. See Theorem \ref{thm:QsmQfactIndex1}(ii) for the precise meaning of 
the ``global index one cover'' in (ii). 

Minagawa (\cite[Main Theorem (2)]{MR1860839}) proved that a weak Fano $3$-fold with Gorenstein terminal singularities 
can be deformed to that with only ordinary double points. 
Moreover, he proved that, if it is $\mathbb{Q}$-factorial, then it admits a smoothing, that is, 
a deformation to a smooth weak Fano $3$-fold (\cite[Main Theorem (3)]{MR1860839}). 
Theorem \ref{thm:wqfanoQsmintro} is a partial generalization of these results. 

\vspace{2mm}

In section \ref{section:stack}, we also treat the canonical covering stack associated to a $3$-fold with terminal singularities or 
a surface with klt singularities. 
We shall explain the unobstructedness of deformations of a $\mathbb{Q}$-Fano $3$-fold 
and $\mathbb{Q}$-Gorenstein deformations of a weak log del Pezzo surface.

\subsection{Comments on the proof} 
	To prove Theorem \ref{thm:wqfano3intro}, we apply the $T^1$-lifting property (\cite{MR1144440}, \cite{MR1144434}, \cite{MR1658200}) to the deformation functor of a pair $(X,D)$ of a weak $\mathbb{Q}$-Fano $3$-fold $X$  
	and a smooth divisor $D \in |{-}mK_X|$ for a sufficiently large $m>0$. 
	The author proved the unobstructedness for a smooth weak Fano variety by using 
	the $E_1$-degeneration of the Hodge to de Rham spectral sequence 
	associated to the log de Rham complex (\cite[Theorem 1.1]{MR3264677}). 
	Since we do not have such a statement on a $3$-fold 
	with terminal singularities, we prove the necessary statements directly by using 
	some arguments similar to those in the proof of \cite[Theorem 1]{MR1286924}. 
	
	To prove Theorem \ref{thm:wqfanoQsmintro}, we use $2$ types of invariants of singularities.
	For (i), we use the invariant $\mu^{(1)}$ which is a version of the invariant $\mu$ 
	used in \cite[Section 2]{MR1358982}. 
	For (ii), we use the coboundary map $\phi_U$ which is used in \cite[Section 1]{MR1358982} and  \cite{MR3692020}. 
	
Finally let us give comments on Theorem \ref{thm:QFano3unobs}. 
Although Theorem \ref{thm:wqfano3intro} is stronger, we shall prove the unobstructedness for a 
	$\mathbb{Q}$-Fano $3$-fold $X$ by associating the canonical covering stack $\mathfrak{X} \rightarrow X$. 
	We shall show that the obstruction space $\Ext^2_{\mathcal{O}_{\mathfrak{X}}}(\Omega^1_{\mathfrak{X}}, \mathcal{O}_{\mathfrak{X}})$ 
	for $\mathfrak{X}$ vanishes. 
	An advantage of this method is that the canonical sheaf $\omega_{\mathfrak{X}}$ is invertible and the $\Ext$ group is the dual of some cohomology group.

\section{Unobstructedness of deformations of a weak $\mathbb{Q}$-Fano $3$-fold}
In this section, we shall prove Theorem \ref{thm:wqfano3intro} about the unobstructedness. 
We first prepare necessary materials on infinitesimal deformations.  

\subsection{Preliminaries on deformation functors}

Let $X$ be an algebraic scheme and $D$ its closed subscheme. 
Let $\Art_{\mathbb{C}}$ be the category of Artinian local $\mathbb{C}$-algebras 
with residue field $\mathbb{C}$.  
We consider the deformation functors $\Def_X, \Def_{(X,D)} \colon \Art_{\mathbb{C}} \rightarrow (Sets)$ 
of $X$ and $(X,D)$ respectively which are defined as follows (cf. \cite[Definition 2.1, 2.2]{MR3419958}). 

First, we introduce the deformation functor of an algebraic scheme. 
\begin{defn}\label{defn:defsch}(cf. \cite[1.2.1]{MR2247603}) 
Let $X$ be an algebraic scheme over $\mathbb{C}$ and $S$ an algebraic scheme over $\mathbb{C}$ with 
a closed point $s \in S$. 
A {\it deformation} of $X$ over $S$ is a pair $(\mathcal{X}, i)$, 
where $\mathcal{X}$ is a scheme flat over $S$ and 
$i \colon X \hookrightarrow \mathcal{X}$ is a closed immersion such that 
the induced morphism $X \rightarrow \mathcal{X} \times_S \{s \}$ is an isomorphism. 

Two deformations $(\mathcal{X}_1, i_1)$ and $(\mathcal{X}_2, i_2)$ over $S$ are said to be 
{\it equivalent} if there exists an isomorphism $\varphi \colon \mathcal{X}_1 \rightarrow \mathcal{X}_2$ over $S$ 
which commutes the following diagram; 
\[
\xymatrix{
X \ar@{^{(}->}^{i_1}[r] \ar@{^{(}->}^{i_2}[rd] & \mathcal{X}_1 \ar[d]^{\varphi} \\
 & \mathcal{X}_2
}
\]
Define the functor $\Def_X \colon \Art_{\mathbb{C}} \rightarrow (Sets)$ by setting 
\begin{equation}\label{functordef}
\Def_X(A):= \{ (\mathcal{X}, i): \text{deformation of } X \text{ over } \Spec A \}/({\rm equiv}),    
\end{equation}
where $({\rm equiv})$ means the equivalence introduced in the above. 
\end{defn} 

We also introduce the deformation functor of a closed immersion. 

\begin{defn}\label{defn:pairdefsch}(cf. \cite[3.4.1]{MR2247603})
Let $f \colon D \hookrightarrow X$ be a closed immersion of algebraic schemes over  $\mathbb{C}$ and $S$ an algebraic scheme over $\mathbb{C}$ with 
a closed point $s \in S$. 
A {\it deformation} of a pair $(X,D)$ over $S$ is a data $(F, i_X, i_D)$ in the cartesian diagram 
\begin{equation}\label{deformmapdiagram}
\xymatrix{
D \ar@{^{(}->}[r]^{i_D} \ar[d]^{f} & \mathcal{D} \ar[d]^{F} \\
X \ar@{^{(}->}[r]^{i_X}  \ar[d] & \mathcal{X} \ar[d]^{\Psi} \\ 
\{s \} \ar@{^{(}->}[r] & S, 
}
\end{equation}
where $\Psi$ and $\Psi \circ F$ are flat and $i_D, i_X$ are closed immersions. 
Two deformations $(F, i_D, i_X)$ and $(F', i_D', i_X')$ of $(X,D)$ over $S$ are said to be 
{\it equivalent} if there exist isomorphisms $\alpha \colon \mathcal{X} \rightarrow \mathcal{X}'$ and 
$\beta \colon \mathcal{D} \rightarrow \mathcal{D}'$ over $S$ 
which commutes the following diagram; 
\[
\xymatrix{
D \ar@{^{(}->}[r]^{i_D} \ar@{^{(}->}[rd]^{i'_{D}} & \mathcal{D} \ar[r] \ar[d]^{\beta} & \mathcal{X} \ar[d]^{\alpha} & 
X \ar@{^{(}->}[l]^{i_X} \ar@{^{(}->}[ld]^{i_X'} \\
 & \mathcal{D}' \ar[r] & \mathcal{X}'. &   
}
\] 
 
We define the functor $\Def_{(X,D)} \colon \Art_{\mathbb{C}} \rightarrow (Sets)$ by setting 
\begin{equation}\label{pairfunctordef}
\Def_{(X,D)}(A):= \{ (F, i_D, i_X): \text{deformation of } (X,D) \text{ over } \Spec A \}/({\rm equiv}),    
\end{equation}
where $({\rm equiv})$ means the equivalence introduced in the above. 
\end{defn}

We use the following notion of a $T^1$-space for the $T^1$-lifting property. 

\begin{defn}\label{defn:T1space} Let $\iota \colon D \hookrightarrow X$ be a closed immersion 
	of algebraic schemes as in Definition \ref{defn:pairdefsch}. 
Set $A_n:= \mathbb{C}[t]/(t^{n+1})$ and $B_n:= \mathbb{C}[x,y]/ (x^{n+1},y^2) \simeq A_n \otimes_{\mathbb{C}} A_1$.
	For $[ (X_n,D_n), \phi_0] \in \Def_{(X,D)}(A_n)$, let $T^1((X_n,D_n)/A_n)$  be the set of isomorphism classes of pairs 
	$((Y_n,E_n), \psi_n) $ consisting of deformations $(Y_n, E_n)$ of $(X_n, D_n)$ over $B_n$ and marking isomorphisms 
	$\psi_n \colon Y_n \otimes_{B_n} A_n 
	\rightarrow X_n$ such that $\psi_n(E_n \otimes_{B_n} A_n) = D_n$, where we use a ring homomorphism 
	$B_n \rightarrow A_n$ given by $x \mapsto t$ and $y \mapsto 0$. 
	We call  $T^1((X_n, D_n)/A_n)$ the {\it $T^1$-space} of $(X_n, D_n)$. 
	\end{defn} 
	
	We have the following interpretation of the $T^1$-space. 
	
	\begin{prop}\label{prop:T1Ext} 
		Let $X$ be a reduced algebraic scheme over $\mathbb{C}$ 
		and $D$ be its divisor such that $D \subset X_{\sm}$, where $X_{\sm} \subset X$ 
		is the smooth locus.  Let $(X_n, D_n) \in \Def_{(X,D)}(A_n)$.  
	
	Then we have 
	\begin{equation}
	T^1((X_n,D_n)/A_n) \simeq \Ext^1_{\mathcal{O}_{X_n}}(\Omega^1_{X_n/A_n}(\log D_n), \mathcal{O}_{X_n}).    
	\end{equation} 
	\end{prop}

	\begin{proof}
	We can prove this by a standard argument 
	 using $B_n \simeq A_n \otimes_{\mathbb{C}} A_1$. We write the proof for the convenience of the reader. 
	
	We shall construct two homomorphisms \[
	\Phi \colon T^1((X_n/D_n)/A_n) \rightarrow 
	\Ext^1_{\mathcal{O}_{X_n}}(\Omega^1_{X_n/A_n}(\log D_n), \mathcal{O}_{X_n}), 
	\]
	\[
	\Psi \colon \Ext^1_{\mathcal{O}_{X_n}}
	(\Omega^1_{X_n/A_n}(\log D_n), \mathcal{O}_{X_n}) 
	\rightarrow  T^1((X_n/D_n)/A_n)
	\] 
	and check that they are converse to each other.  
	
	Given $\xi_n= (Y_n, E_n) \in T^1((X_n/D_n)/A_n)$,  
	 consider the exact sequence 
	\[
	0 \rightarrow \mathcal{O}_{X_n} \rightarrow 
	\Omega^1_{Y_n/A_n} (\log E_n)|_{X_n} \rightarrow 
	\Omega^1_{X_n/A_n} (\log D_n) \rightarrow 0 
	\]
	and let $\Phi(\xi_n) \in  \Ext^1_{\mathcal{O}_{X_n}}
	(\Omega^1_{X_n/A_n}(\log D_n), \mathcal{O}_{X_n})$ 
	be the class corresponding to the above extension. 
	This defines a homomorphism $\Phi$.  
	
	Conversely, given an extension  
	\[
	0 \rightarrow \mathcal{O}_{X_n} \xrightarrow{\beta} 
	\mathcal{F} \xrightarrow{\alpha} 
	\Omega^1_{X_n/A_n} (\log D_n) \rightarrow 0 
	\]
	corresponding to a class $\eta \in \Ext^1_{\mathcal{O}_{X_n}}
	(\Omega^1_{X_n/A_n}(\log D_n), \mathcal{O}_{X_n})$. 
	This can be pulled back by $d \colon \mathcal{O}_{X_n} 
	\rightarrow \Omega^1_{X_n/A_n}(\log D_n)$ to an extension 
	of $A_n$-algebras 
	\[
	0 \rightarrow \mathcal{O}_{X_n} \rightarrow 
	\mathcal{A} \rightarrow \mathcal{O}_{X_n} 
	\rightarrow 0, 
	\]
	where we put $\mathcal{A}:= \mathcal{F} \times_{\Omega^1_{X_n/A_n}(\log D_n)} 
	\mathcal{O}_{X_n}$. 
	On a open subset $U \subset X$, the sections of $\mathcal{A}$ is 
	\[
	\mathcal{A}(U) = \{ (e, g) \in \mathcal{F}(U) \times
	\mathcal{O}_{X_n}(U) \mid \alpha(e) = d g \}.
	\]
	The algebra structure on $\mathcal{A}$ is determined by 
	\[
	(e_1, g_1)\cdot (e_2, g_2) = (g_1e_2 + g_2 e_1, g_1 g_2)
	\]
	for $(e_1, g_1), (e_2, g_2) \in \mathcal{A}(U)$ on a open set $U$. 
	Moreover the $B_n$-algebra structure on $\mathcal{A}$ is determined by a homomorphism $B_n \rightarrow \mathcal{A}$ sending $y$ to $(\beta(1), 0) \in \mathcal{A}$, where $y \in B_n$ satisfies $y^2 =0$ in the definition.  
Then we see that $\mathcal{A}$ is flat over $B_n$ by the local criterion of flatness. 
We also see that $\mathcal{A} \otimes_{B_n} A_n \simeq \mathcal{O}_{X_n}$, 
thus, by putting $\mathcal{O}_{Y_n}:= \mathcal{A}$, we obtain a deformation of $X_n$ 
over $B_n$.  

We define an ideal sheaf $\mathcal{I} \subset \mathcal{A}$ by 
\[
\mathcal{I} := \mathcal{F}(-D_n) \times_{\Omega^1_{X_n/A_n}(\log D_n)(-D_n)} 
	\mathcal{O}_{X_n}(-D_n).  
\] 	
Note that we have a natural commutative diagram 
\[
\xymatrix{
\mathcal{O}_{X_n} \ar[r]^{d}  & \Omega^1_{X_n/A_n}(\log D_n)  \\
\mathcal{O}_{X_n}(-D_n) \ar[r]^{d \ \ \ \ \ \ } \ar@{^{(}->}[u]& \Omega^1_{X_n/A_n}(\log D_n)(-D_n).  
\ar@{^{(}->}[u]  
}
\]

We can check that $\mathcal{I}$ is an ideal sheaf of $\mathcal{A}$. 
Thus $\mathcal{I}$ has a $B_n$-module structure and   
it fits in a commutative diagram whose horizontal sequences are exact: 
\[\xymatrix{ 
0 \ar[r] & \mathcal{O}_{X_n}(-D_n) \ar[r]  &
	\mathcal{F}(-D_n)  \ar[r]  & 
	\Omega^1_{X_n/A_n}(\log D_n)(-D_n)  
	\ar[r]  & 0  \\
0 \ar[r] & \mathcal{O}_{X_n}(-D_n) \ar[r] \ar[u]^{=} &
	\mathcal{I}  \ar[r] \ar[u] & 
	\mathcal{O}_{X_n}(-D_n)  
	\ar[r] \ar[u]^{d} & 0. 
	}
\] 
Hence $\mathcal{I}$ is flat over $B_n$ again by the local criterion of flatness. 
We can also check that $\mathcal{I} \otimes_{B_n} A_n \simeq \mathcal{O}_{X_n}(-D_n)$. 
We can define a lifting $E_n \subset Y_n$ of $D_n \subset X_n$ by putting 
$\mathcal{I}_{E_n}:= \mathcal{I}$. 

By using $Y_n, E_n$, we can define $\Psi (\eta):= [(Y_n, E_n)] \in T^1((X_n, D_n)/A_n)$. 
Thus we obtain $\Psi$. 

We can check that $\Phi$ and $\Psi$ are converse to each other and obtain the required isomorphism. 
	\end{proof}

The following proposition shows that a ${\rm S}_3$ sheaf behaves well under base change and 
that infinitesimal deformations of a ${\rm S}_3$ scheme do not change by removing a codimension $3$ 
closed subset.

\begin{prop}\label{prop:S3flat} 
Let $X$ be an algebraic scheme over $\mathbb{C}$ and $\iota \colon U  \hookrightarrow X$ 
be an open immersion. 
Assume that $\depth \mathcal{O}_{X,p} \ge 3$ for all scheme theoretic point $p \in X \setminus U$. 
Let  $f_{U_A} \colon U_A \rightarrow \Spec A$ be a deformation of $U$ over $(A, \mathfrak{m}_A) \in \Art_{\mathbb{C}}$.  
Let $\mathcal{F}_{U_A}$ be an $A$-flat coherent sheaf on $U_A$.  
Let $\mathcal{F}_{A}:= \iota_*\mathcal{F}_{U_A}$ be a sheaf of $A$-modules on a ringed space $(X, \iota_* \mathcal{O}_{U_A})$ 
and $\mathcal{F}:= \iota_*( \mathcal{F}_{U_A} \otimes_A A/ \mathfrak{m}_A)$ be an $\mathcal{O}_X$-module. 
Assume that, on the stalk,  
\begin{equation}
\depth_{\mathcal{O}_{X,p}} \mathcal{F}_p \ge 3 \tag{*}
\end{equation}
for all scheme theoretic point $p \in X \setminus U$. 

 Let $M$ be a finite $A$-module. Then we have the following. 
 \begin{enumerate}
 \item [(i)]
 \[
R^1 \iota_* (\mathcal{F}_{U_A} \otimes_{A} M ) =0.  
\]
 \item[(ii)] 
 \[
 \mathcal{F}_A \otimes_{A} M \simeq \iota_* (\mathcal{F}_{U_A} \otimes_A M).
 \]
 \item[(iii)] The ringed space $X_A:= (X, \iota_* \mathcal{O}_{U_A})$ is an algebraic scheme, flat over $A$ and 
 $\mathcal{O}_{X_A} \otimes_A \mathbb{C} \simeq \mathcal{O}_X$. 
 Thus it defines $X_A \in \Def_X(A)$.  \\
 Moreover, the sheaf of $A$-modules $\mathcal{F}_A$ is a coherent $\mathcal{O}_{X_A}$-module, flat over $A$ and satisfies $\mathcal{F}_A \otimes_A \mathbb{C} \simeq \mathcal{F}$. 
 \end{enumerate} 
\end{prop}

\begin{proof}
\noindent(i)  We prove this by induction on $\dim_{\mathbb{C}} M$. 

If $\dim_{\mathbb{C}} M = 1$, then we have $M \simeq \mathbb{C}$. 
Thus we have $R^1 \iota_* (\mathcal{F}_{U_A} \otimes_{A} \mathbb{C}) =0$ 
by the condition (*) on the depth. 

We consider the general case. 
Note that there exists $m \in M$ such that $A \cdot m \simeq \mathbb{C}$. 
Thus we have an exact sequence 
\begin{equation}\label{eq:exactmodule}
0 \rightarrow \mathbb{C} \rightarrow M \rightarrow M' \rightarrow 0.  
\end{equation}
By tensoring this with $\mathcal{F}_{U_A}$, and taking $R^1 \iota_*$, 
we obtain an exact sequence 
\[
 R^1 \iota_* ( \mathcal{F}_{U_A} \otimes_A \mathbb{C}) 
\rightarrow R^1 \iota_* ( \mathcal{F}_{U_A} \otimes_A M )
\rightarrow R^1 \iota_* ( \mathcal{F}_{U_A} \otimes_A M' ).   
\]
By the induction hypothesis, the both sides are zero. 
Hence we see that the middle term is also zero. 

\vspace{2mm}

\noindent(ii) 
The identity homomorphism $\iota^* (\mathcal{F}_A \otimes_A M) \rightarrow \mathcal{F}_{U_A} \otimes_A M$ 
defines a homomorphism \[
\phi \colon \mathcal{F}_A \otimes_A M \rightarrow \iota_* (\mathcal{F}_{U_A} \otimes_A M). 
\]


We shall show that $\phi$ is an isomorphism.  

\noindent({\bf Case 1}) We first treat the case where $M$ admits a surjection $A \rightarrow M$ and let $A':= M$.  
We may assume that $\pi \colon A \rightarrow A'$ is a small extension, that is, $\dim_{\mathbb{C}} \Ker \pi =1$ 
since we can decompose $\pi$ into small extensions. 
Let $(t):= \Ker \pi$ with an exact sequence 
\begin{equation}\label{eq:smallext}
0 \rightarrow (t) \rightarrow A \rightarrow A' \rightarrow 0. 
\end{equation}
By tensoring this with $\mathcal{F}_A$ and $\mathcal{F}_{U_A}$, and taking $\iota_*$, 
 we obtain a commutative diagram 
\begin{equation}\label{eq:twoexact}
\xymatrix{ 
& \mathcal{F}_A \otimes_A (t) \ar[r] \ar[d]^{\phi'} & 
 \mathcal{F}_A \ar[r] \ar[d]^{\simeq} & 
 \mathcal{F}_A \otimes_A A' \ar[r] \ar[d]^{\phi} & 
 0 \\
0 \ar[r] &  \iota_* (\mathcal{F}_{U_A} \otimes_A (t)) \ar[r]  & 
  \iota_* ( \mathcal{F}_{U_A} )
\ar[r]  & 
  \iota_* ( \mathcal{F}_{U_A} \otimes_A A' ) 
\ar[r]  & 0.    
}
\end{equation}
Both horizontal sequences are exact. 
The exactness of the bottom sequence follows from 
 the flatness of $\mathcal{F}_{U_A}$ and $R^1 \iota_* \mathcal{F}_U =0$ by (i). 
 The vertical homomorphisms $\phi$ and $\phi'$ are surjective. 
 Indeed, the surjectivity of $\phi'$ follows from (i) again since we have $(t) \simeq \mathbb{C} \simeq A/ \mathfrak{m}_A$ 
 for the maximal ideal $\mathfrak{m}_A$. 
 Hence, by the diagram chasing, we see that $\phi$ is injective. 
 Thus we see that $\phi$ is an isomorphism. 
 
\vspace{2mm} 
 
\noindent({\bf Case 2}) In general case, we show that $\phi$ is an isomorphism by induction on $\dim_{\mathbb{C}}M$. 
 When $\dim_{\mathbb{C}} M =1$, then the statement is proved in (Case 1). 
  
  For $m \in M$ such that $A \cdot m \subset M$ is $1$-dimensional over $\mathbb{C}$, we have an 
 exact sequence as in (\ref{eq:exactmodule}). 
Thus we have the following two exact sequences with commutative diagrams:  
 \begin{equation}\label{eq:twoexactmodules}
\xymatrix{ 
& \mathcal{F}_A \otimes_A \mathbb{C} \ar[r] \ar[d]^{\phi_0} & 
 \mathcal{F}_A \otimes_A M \ar[r] \ar[d]^{\phi} & 
 \mathcal{F}_A \otimes_A M' \ar[r] \ar[d]^{\phi_{M'}} & 
 0 \\
0 \ar[r] &  \iota_* (\mathcal{F}_{U_A} \otimes_A \mathbb{C}) \ar[r]  & 
  \iota_* ( \mathcal{F}_{U_A} \otimes_A M)
\ar[r]  & 
  \iota_* ( \mathcal{F}_{U_A} \otimes_A M' ) 
\ar[r]  & 0.    
}
\end{equation}
The induction hypothesis implies that $\phi_{M'}$ is an isomorphism. 
Since $\phi_0$ is also an isomorphism, we see that $\phi$ is also an isomorphism by the diagram.  
  
 \vspace{2mm}
 
 \noindent(iii) 
For the former statement, we may assume that $X$ is affine and shall show that 
$\iota_* \mathcal{O}_{U_A}$ is isomorphic to the sheaf $\mathcal{O}_{ \Spec H^0( \mathcal{O}_{U_A})}$. 
We prove this by induction on $\dim_{\mathbb{C}} A$. 

Consider a small extension as in (\ref{eq:smallext}). 
By $R^1 \iota_* \mathcal{O}_U=0$, we have an exact sequence 
\[
0 \rightarrow \iota_* \mathcal{O}_U \rightarrow \iota_* \mathcal{O}_{U_A} 
\rightarrow \iota_*( \mathcal{O}_{U_A} \otimes_A A') \rightarrow 0.
\]
 The both side terms are isomorphic to the sheaf coming from the corresponding affine schemes. 
 Hence the middle one is also isomorphic to the sheaf of an affine scheme. 
 Thus we see that the ringed space $X_A =(X, \iota_* \mathcal{O}_{U_A})$ is an algebraic scheme. 
 We may also check that $\mathcal{F}_A$ is coherent by the exact sequence (\ref{eq:twoexact}). 
 
 Next we shall check the $A$-flatness of $\mathcal{O}_{X_A}$ and $\mathcal{F}_A$. 
 By (ii), we see that $\phi$ and $\phi'$ are isomorphisms in the diagram (\ref{eq:twoexact}). 
 By induction on $\dim_{\mathbb{C}}A$, 
  we shall show the $A$-flatness of $\mathcal{F}_A$. 
  
 When $\dim_{\mathbb{C}}A =1$, then the flatness is trivial. 
 Let $A \rightarrow A'$ be a small extension as in (\ref{eq:smallext}). 
 We see that $\mathcal{F}_{A'} \simeq \mathcal{F}_A \otimes_A A'$ is flat over $A'$ 
 by the induction hypothesis.  
 By the local criterion of flatness (cf. \cite[Proposition 2.2]{MR2583634}) and 
 the exact sequence in (\ref{eq:twoexact}), we see that 
 $\mathcal{F}_A$ is flat over $A$. 
 
In the same way, we can check that $\mathcal{O}_{X_A}$ is flat over $A$.  
\end{proof}

\begin{rem}
A similar proposition was treated in \cite[Proposition 12]{MR1339664}. 
However the original statement has a counterexample (\cite[Example 3.12]{Lee:2016aa}). 
Lee and Nakayama also treated similar propositions (cf. \cite[Proposition 3.7, Corollary 3.10, Lemma 3.14.]{Lee:2016aa}) 
which should imply Proposition \ref{prop:S3flat}. 
However, we include the proof of Proposition \ref{prop:S3flat} to make this paper self-contained.   
\end{rem}

\vspace{2mm} 

As a corollary, we have the following proposition about the behaviour of the deformation functor 
of a klt variety under removing a codimension $3$ closed subscheme.

\begin{cor}\label{cor:kltrestrictisom}
Let $X$ be a normal variety with only klt singularities and 
$D$ be an effective $\mathbb{Q}$-Cartier Weil divisor on $X$. 
Let $Z \subset X$ be a closed subscheme whose irreducible components $Z_i$ satisfy $\codim_X Z_i \ge 3$. 
Let $U:= X \setminus Z$ and $D_U:= D \setminus Z$. 
Consider the restriction morphism of functors 
\[
\Phi \colon \Def_{(X,D)} \rightarrow \Def_{(U, D_U)}. 
\]
Then we have the following. 
\begin{enumerate}
\item[(i)]  There is an inverse $\Psi \colon \Def_{(U,D_U)} \rightarrow \Def_{(X,D)}$ of $\Phi$. 
Hence $\Phi$ is an isomorphism. 
\item[(ii)] Let $(X_n, D_n) \in \Def_{(X,D)}(A_n)$ over $A_n$ as in Definition \ref{defn:T1space} 
and its restriction $(U_n, D_{U_n}) \in \Def_{(U, D_U)}(A_n)$ by $\Phi$. 
Then we have an isomorphism of the $T^1$-spaces 
\[
T^1((X_n, D_n)/A_n) \simeq T^1((U_n, D_{U_n})/A_n). 
\] 
\end{enumerate}
\end{cor}

\begin{proof}
(i) Let $(U_A, D_{U_A}) \in \Def_{(U, D_U)}(A)$ over $A \in \Art_{\mathbb{C}}$. 
We shall construct $\Psi((U_A,D_{U_A}))=(X_A, D_A)$ as follows. 

Let $\mathcal{O}_{X_A}:= \iota_* \mathcal{O}_{U_A}$. 
By Proposition \ref{prop:S3flat}(iii), we see that the ringed space 
$X_A:=(X, \mathcal{O}_{X_A})$ is an algebraic scheme such that $X_A \in \Def_X(A)$. 

Next define an ideal sheaf $\mathcal{I}_{D_A} \subset \mathcal{O}_{X_A}$ by  
$\mathcal{I}_{D_A}:= \iota_* \mathcal{I}_{D_{U_A}} $. 
Since $\mathcal{I}_D= \mathcal{O}_X(-D)$ is Cohen-Macaulay (cf. \cite[Corollary 5.25]{MR1658959}), 
the sheaf $\mathcal{I}_{D_A}$ satisfies the condition of $\mathcal{F}_A$ in Proposition \ref{prop:S3flat}. 
Thus $\mathcal{I}_{D_A}$ is flat over $A$ and the corresponding closed subscheme $D_A \subset X_A$ 
defines an element $(X_A, D_A) \in \Def_{(X,D)}(A)$.  

For a homomorphism $A \rightarrow A'$ in $\Art_{\mathbb{C}}$, 
 the compatibility 
 \[
 \Psi ((U_A, D_{U_A}) \otimes_A A') = \Psi ((U_A, D_{U_A})) \otimes_{A} A'
 \]
  follows from 
 Proposition \ref{prop:S3flat} (ii). 
  Hence we have a natural transformation $\Psi \colon \Def_{(U,D_U)} \rightarrow \Def_{(X,D)}$. 
  
  We can check that $\Phi$ and $\Psi$ are converse to each other by Proposition \ref{prop:S3flat}. 
  
 \noindent(ii) This immediately follows from (i). 
\end{proof}

For a deformation of a $3$-fold with only terminal singularities, we also have the following lemma about $\Ext$ group under restriction to a open subset 
over $A_n$.  

\begin{lem}\label{lem:Extopenisom}(cf. \cite[Lemma 2.11]{MR3653082}) 
	Let $X$ be a $3$-fold with only terminal singularities and $\xi_n:= (X_n \rightarrow \Spec A_n) \in \Def_X(A_n)$. Let $U \subset X$ be the smooth locus of $X$ and 
	$U_n \rightarrow \Spec A_n$ be a deformation of $U$ induced by $\xi_n$. 
	Let $\mathcal{F}, \mathcal{L}$ be reflexive sheaves on $X_n$. Assume that $R^1 \iota_* \mathcal{L}|_{U_n} =0$, where $\iota \colon U \hookrightarrow X$ is an open immersion. 
	
	Then the restriction homomorphism 
	\begin{equation}
	r \colon \Ext^1_{\mathcal{O}_{X_n}} (\mathcal{F}, \mathcal{L}) 
	\rightarrow \Ext^1_{\mathcal{O}_{U_n}} (\mathcal{F}|_{U_n}, \mathcal{L}|_{U_n}) 
	\end{equation}
	is an isomorphism. 
	\end{lem}

\begin{proof}
	We can construct the converse of $r$ as follows. Given an exact sequence
	\[
	0 \rightarrow \mathcal{L}|_{U_n} \rightarrow \mathcal{G}_{U_n} \rightarrow 
	\mathcal{F}|_{U_n} \rightarrow 0, 
	\]
	its push-forward 
	\[
	0 \rightarrow \mathcal{L} \rightarrow \iota_* \mathcal{G}_{U_n} 
	\rightarrow \mathcal{F} \rightarrow 0 
	\]
	is also exact by the condition $R^1 \iota_* \mathcal{L}|_{U_n} =0$. 
	
	\end{proof}

We also need the following proposition about the obstruction to smoothness of the forgetful morphism. 

\begin{prop}\label{prop:forgetfulsmooth}
Let $X$ be an algebraic scheme and $D$ be its closed subscheme. 
Let $\Def_{(X,D)}$ and $\Def_X$ be the deformation functors defined in Definitions \ref{defn:defsch}, \ref{defn:pairdefsch} 
and 
\[
\Phi \colon \Def_{(X, D)} \rightarrow \Def_X
\] be the forgetful morphism. 
Assume the following two conditions:  
\begin{enumerate}
\item[(i)] There is a Zariski open cover $\{U_i \}_{i=1}^n$ such that 
the forgetful morphism $\Def_{(U_i, D_i)} \rightarrow \Def_{U_i}$ is a smooth morphism for $i=1, \ldots, n$, 
where $D_i:= D \cap U_i$. 
\item[(ii)] $H^1(D, \mathcal{N}_{D/X})=0$, where $\mathcal{N}_{D/X}$ is the normal sheaf. 
\end{enumerate}  
Then $\Phi$ is a smooth morphism. 
\end{prop}

\begin{rem}
The condition (i) holds when $D \subset X$ is a l.c.i. subscheme (cf.  \cite[Remark 6.2.1]{MR2583634}). 
\end{rem}

\begin{proof}
This can be deduced from \cite[Theorem 6.2 (b)]{MR2583634} as follows. 
Given an extension $0 \rightarrow J \rightarrow A' \rightarrow A \rightarrow 0$ in $\Art_{\mathbb{C}}$ 
such that $\mathfrak{m}_{A'} \cdot J=0$, $(X_A, D_A) \in \Def_{(X,D)}(A)$ and 
$X_{A'} \in \Def_{X}(A')$ such that $X_{A'} \otimes_{A'} A \simeq X_A$. 
Since extensions of $D_A$ over $A'$ exist locally on $X_A$ by the condition (i), 
we can apply \cite[Theorem 6.2 (b)]{MR2583634} to define an obstruction class $o_{(X_A, D_A), X_{A'}} \in H^1(D, \mathcal{N}_{D/X} \otimes_{\mathbb{C}} J)$ 
for the existence of a lifting $(X_{A'}, D_{A'}) \in \Def_{(X,D)}(A')$ of $(X_A,D_A)$.  
Thus $H^1(D, \mathcal{N}_{D/X})$ is an obstruction space for smoothness of $\Phi$ 
and, by the condition (ii), the smoothness of $\Phi$ follows. 
\end{proof}

\vspace{2mm}

\subsection{Proof of Theorem \ref{thm:wqfano3intro}}

The following result is crucial to prove Theorem \ref{thm:wqfano3intro}.

\begin{thm}\label{thm:pairunobs}
	Let $X$ be a weak $\mathbb{Q}$-Fano $3$-fold. 
	Take a smooth member  $ D \in |{-}mK_X|$ for some positive integer $m$ such that 
	$D \cap \Sing X = \emptyset$ which exists by the base point free theorem. 
	
	Then the deformation functor $\Def_{(X,D)}$ is unobstructed. 
	\end{thm}

\begin{proof}
	We shall use the $T^1$-lifting property. The proof is similar to \cite[Theorem 2.2]{MR3264677}, 
	but also follows some arguments in \cite[Section 2]{MR1286924}.

	By the $T^1$-lifting theorem (\cite[Theorem A]{MR1658200}) and Proposition \ref{prop:T1Ext}, it is enough to show that the base-change  homomorphism 
	\[
	\Ext^1_{\mathcal{O}_{X_n}}(\Omega^1_{X_n/A_n}(\log D_n), \mathcal{O}_{X_n}) \rightarrow \Ext^1_{\mathcal{O}_{X_{n-1}}}(\Omega^1_{X_{n-1}/A_{n-1}}(\log D_{n-1}), \mathcal{O}_{X_{n-1}})
	\]
	is surjective. 
	
	Let $U$ be the smooth locus of $X$ and $\iota \colon U \hookrightarrow X$ 
	be the open immersion. Note that $D \subset U$. 
	Let $(U_n,D_n) \in \Def_{(U,D)}(A_n)$ be the deformation induced by $(X_n, D_n)$.  
	 We have two isomorphisms   
	\begin{multline}
	\Ext^1_{\mathcal{O}_{X_n}}(\Omega^1_{X_n/A_n}(\log D_n), \mathcal{O}_{X_n}) \simeq 
	\Ext^1_{\mathcal{O}_{U_n}}(\Omega^1_{U_n/A_n}(\log D_n), \mathcal{O}_{U_n}) \\
	\simeq 
	\Ext^1_{\mathcal{O}_{X_n}} \left( \iota_* \left( \Omega^1_{U_n/A_n}(\log D_n) \otimes \omega_{U_n/A_n} \right), \omega_{X_n/A_n} \right),  
	\end{multline}
	where $\iota \colon U_n \hookrightarrow X_n$ is an open immersion. 
	The former isomorphism follows from Corollary \ref{cor:kltrestrictisom} (ii). 
	The latter isomorphism follows from Lemma \ref{lem:Extopenisom} since $\omega_X$ is Cohen-Macaulay and
	 we have $R^1 \iota_* \omega_{U_n/ A_n} =0$ 
	by applying Proposition \ref{prop:S3flat}(i) in the case $M=A= A_n$  (cf. \cite[Claim 2.12]{MR3419958}). 
	By the Serre duality, we obtain 
	\begin{multline}
	\Ext^1_{\mathcal{O}_{X_n}}(\iota_* \left( \Omega^1_{U_n/A_n}(\log D_n) \otimes \omega_{U_n/A_n} \right), \omega_{X_n/A_n}) \\
	 \simeq  \Hom_{A_n} \left(	H^2 \left( X_n, \iota_* \left( \Omega^1_{U_n/A_n}(\log D_n) \otimes \omega_{U_n/A_n} \right) \right), A_n \right).  
	\end{multline}
	
	Thus it is enough to show the homomorphism 
	\begin{multline}
	\Hom_{A_n} (	H^2(X_n, \iota_* \left( \Omega^1_{U_n/A_n}(\log D_n) \otimes \omega_{U_n/A_n} \right)), A_n) \\
	\rightarrow \Hom_{A_{n-1}} \left(	H^2 \left( X_{n-1}, \iota_* \left( \Omega^1_{U_{n-1}/A_{n-1}}(\log D_{n-1}) \otimes \omega_{U_{n-1}/A_{n-1}} \right) \right), A_{n-1} \right)
	\end{multline}
is surjective. 

Let $\pi_n \colon Z_n \rightarrow X_n$ be a cyclic cover branched along $D_n \in |{-}m K_{X_n/A_n}|$. 
Now let $M^{\vee}$ be the dual of  some object $M$.  
We have an isomorphism 
\[
\Omega^1_{Z_n}(\log \Delta_n) \simeq (\pi_n^* \Omega^1_{X_n}(\log D_n))^{\vee \vee}
\]
for some divisor 
 $\Delta_n \in |{-}\pi_n^* K_{X_n / A_n}|$, 
 where $\pi_n^* K_{X_n/ A_n}$ is a Cartier divisor on $Z_n$ 
 corresponding to a line bundle $(\pi_n^* \omega_{X_n/A_n})^{\vee \vee}$. 
Then we obtain the decomposition  
\[
(\pi_n)_* \Omega^1_{Z_n}(\log \Delta_n)(-\Delta_n) \simeq \bigoplus_{i=0}^{m-1} \iota_* (\Omega^1_{U_n}(\log D_n)((i+1) K_{U_n/A_n})). 
\]
By this decomposition, the surjectivity is reduced to that of the homomorphism 

\begin{multline}
	\Hom_{A_n} (	H^2(Z_n, \Omega^1_{Z_n/A_n}(\log \Delta_n)(-\Delta_n)), A_n) \\
	\rightarrow \Hom_{A_{n-1}} (	H^2(Z_{n-1},  \Omega^1_{Z_{n-1}/A_{n-1}}(\log \Delta_{n-1}) (-\Delta_{n-1}) ), A_{n-1}). 
\end{multline}

As in the proof of \cite[Section 2, Theorem 1]{MR1286924}, the above surjectivity is reduced to the following two statements; 
\begin{enumerate}
	\item[(i)]  $H^2(Z_n,  \Omega^1_{Z_n/A_n}(\log \Delta_n)(- \Delta_n))$ is a free $A_n$-module. 
	\item[(ii)] The reduction homomorphism \[
	H^2(Z_n, \Omega^1_{Z_n/A_n}(\log \Delta_n)(-\Delta_n)) \rightarrow 
	H^2(Z_{n-1}, \Omega^1_{Z_{n-1}/A_{n-1}}(\log \Delta_{n-1})(-\Delta_{n-1}))
	\] 
	is surjective. 
	\end{enumerate}

Since $Z_n \rightarrow \Spec A_n$ is a relative l.c.i.\ morphism, we see that $\Omega^1_{Z_n/A_n}(\log \Delta_n)(- \Delta_n)$ is a flat $A_n$-module (cf. \cite[Theorem D.2.7]{MR2247603}). 
Hence we have an exact sequence 
\[
0 \rightarrow \Omega^1_{Z}(\log \Delta)(- \Delta) \rightarrow \Omega^1_{Z_n/A_n}(\log \Delta_n)(- \Delta_n) 
\rightarrow \Omega^1_{Z_{n-1}/A_{n-1}}(\log \Delta_{n-1})(- \Delta_{n-1}) \rightarrow 0. 
\]
The following lemma implies (ii). 

\begin{lem}\label{lem:H3vanish}
	Let $Z:= Z_0$, $\Delta:= \Delta_0$ and $\pi:= \pi_0 \colon Z \rightarrow X$ be as above. Then we have 
	\[
	H^3(Z, \Omega^1_{Z}(\log \Delta)(-\Delta)) =0. 
	\]
	\end{lem}

\begin{proof}[Proof of Lemma \ref{lem:H3vanish}]
Since $Z$ is smooth along $\Delta$, there is an exact sequence 
\[
0 \rightarrow \Omega^1_Z(\log \Delta)(- \Delta) \rightarrow 
\Omega^1_Z \rightarrow \Omega^1_{\Delta} \rightarrow 0. 
\]
This sequence induces an exact sequence 
\[
H^2(\Delta, \Omega^1_{\Delta}) \rightarrow H^3(Z, \Omega^1_Z(\log \Delta)(- \Delta)) 
\rightarrow H^3(Z, \Omega^1_Z). 
\]
Thus it is enough to show that the both side terms are zero. 
	
	We first show $H^2(\Delta, \Omega^1_{\Delta}) =0$. 
	By the Serre duality and Hodge symmetry, 
	we see that $h^2(\Delta, \Omega^1_{\Delta}) = h^0(\Delta, \Omega^1_{\Delta}) = h^1(\Delta, \mathcal{O}_{\Delta})$. 
	There is an exact sequence 
	\[
	H^1(Z, \mathcal{O}_Z) \rightarrow H^1(\Delta, \mathcal{O}_{\Delta}) \rightarrow H^2(Z, \mathcal{O}_Z(- \Delta)).  
	\]
	Since we have $\pi_* \mathcal{O}_Z \simeq \bigoplus_{i=0}^{m-1} \mathcal{O}_X(i K_X)$, we see that 
	$H^1(Z, \mathcal{O}_Z) =0$ by the Kawamata-Viehweg vanishing theorem since $-K_X$ is nef and big. 
Since $\Delta$ is nef and big, we also see that $H^2(Z, \mathcal{O}_Z(- \Delta)) =0$. Thus we obtain $H^2(\Delta, \Omega^1_{\Delta}) =0$. 

Next we show $H^3(Z, \Omega^1_Z) =0$. By the Serre duality, we see that $h^3(Z, \Omega^1_Z) = h^0(Z, \Omega^2_Z)$. 
Let $\mu \colon \tilde{Z} \rightarrow Z$ be a log resolution of singularities of $Z$. 
Since $Z$ has only Gorenstein terminal singularities, we see that $\mu_* \Omega^2_{\tilde{Z}} \simeq \Omega^2_Z$ (cf. \cite[Theorem 4]{MR1863856}). 
	 Hence we obtain $h^0(Z, \Omega^2_Z) = h^0(\tilde{Z}, \Omega^2_{\tilde{Z}}) = h^2(\tilde{Z}, \mathcal{O}_{\tilde{Z}})$, 
	 where we use the Hodge symmetry on $\tilde{Z}$ for the latter equality. 
	 Since $Z$ has only rational singularities, $h^2(\tilde{Z}, \mathcal{O}_{\tilde{Z}}) = h^2(Z, \mathcal{O}_Z)$. 
	 This is zero by the Kawamata-Viehweg vanishing theorem. Thus we obtain $H^3(Z, \Omega^1_Z) =0$. Concluding the proof of Lemma \ref{lem:H3vanish}. 
	\end{proof}

Thus we obtain (ii). 

In order to obtain (i), it is enough to show the surjectivity of the homomorphism 
\[
\Phi_n \colon	H^1(Z_n, \Omega^1_{Z_n/A_n}(\log \Delta_n)(-\Delta_n)) \rightarrow 
	H^1(Z_{n-1}, \Omega^1_{Z_{n-1}/A_{n-1}}(\log \Delta_{n-1})(-\Delta_{n-1}))
\]
by (ii) and the base change theorem (cf. \cite[Theorem 12.11 (b)]{Hartshorne}). 

Let $\mathcal{K}_{(Z_n, \Delta_n)}:= \Ker (\mathcal{O}_{Z_n}^* \rightarrow \mathcal{O}_{\Delta_n}^*)$. 
Then we have a commutative diagram 
\[
\xymatrix{
0 \ar[r] & H^1(Z, \mathcal{K}_{(Z,\Delta)}) \otimes_{\mathbb{Z}} \mathbb{C} \ar[r] \ar[d]^{\alpha_1} & 
 H^1(Z, \mathcal{O}_Z^*) \otimes_{\mathbb{Z}} \mathbb{C}   \ar[r] \ar[d]^{\alpha_2} & 
 H^1(\Delta, \mathcal{O}_{\Delta}^*) \otimes_{\mathbb{Z}} \mathbb{C} \ar[d]^{\alpha_3}  
  \\ 
0 \ar[r] & H^1(Z, \Omega^1_Z(\log \Delta)(- \Delta)) \ar[r] & 
H^1(Z, \Omega^1_Z) \ar[r] & H^1(\Delta, \Omega^1_{\Delta}),  
}
\]
where the vertical homomorphisms are induced by the homomorphisms 
\[
\frac{1}{2 \pi \sqrt{-1}} d\log_Z \colon \mathcal{O}_Z^* \rightarrow \Omega^1_Z, \ \ 
\frac{1}{2 \pi \sqrt{-1}} d \log_{\Delta} \colon \mathcal{O}_{\Delta}^* \rightarrow \Omega^1_{\Delta}.
\]  
Note that we have the upper exact sequence after tensoring $\mathbb{C}$ since $H^1(Z, \mathcal{O}_Z) =0$ by the proof of Lemma \ref{lem:H3vanish} and thus $H^1(Z, \mathcal{O}_{Z}^*)$ is a finitely generated 
$\mathbb{Z}$-module.    
By using arguments in \cite[Lemma 2.2]{MR1286924}, we obtain the following claim. 
\begin{claim} 
The homomorphism $\alpha_2$ is surjective. 
\end{claim}
 
 \begin{proof}[Proof of Claim] 
Note that $Z$ has only isolated cDV singularities by the construction of $Z$. 
Let $\nu \colon \tilde{Z} \rightarrow Z$ be a log resolution of $Z$ such that 
$\nu^{-1}(Z^{\rm sm}) \rightarrow Z^{\rm sm}$ is an isomorphism. 
We have a natural homomorphism $\Omega^1_Z \rightarrow \nu_* \Omega^1_{\tilde{Z}}$. 
This is an isomorphism since we have $H^i_{\Sing Z} (Z, \Omega^1_{Z}) =0$ 
for $i=0, 1$ by \cite[Lemma 2.1]{MR1286924}. 
The Leray spectral sequence induces a commutative diagram 
\begin{equation}\label{eq:diagramsnake}
\xymatrix{
0 \ar[r] & H^1(Z, \mathcal{O}_Z^*)  \otimes_{\mathbb{Z}} \mathbb{C} \ar[r] \ar[d]^{\alpha_2} & 
 H^1(\tilde{Z}, \mathcal{O}_{\tilde{Z}}^*) \otimes_{\mathbb{Z}} \mathbb{C}   \ar[r] \ar[d]^{\tilde{\alpha}_2} & 
 H^0(Z,  R^1 \nu_* \mathcal{O}_{\tilde{Z}}^*) \otimes_{\mathbb{Z}} \mathbb{C} \ar[d]^{\alpha'_2}  
  \\ 
0 \ar[r] & H^1(Z, \Omega^1_Z) \ar[r] & 
H^1(\tilde{Z}, \Omega^1_{\tilde{Z}}) \ar[r] & H^1(Z, R^1 \nu_*\Omega^1_{\tilde{Z}})  
}
\end{equation}
whose horizontal sequences are exact. 
The exactness of the upper sequence follows since 
$H^1(\tilde{Z}, \mathcal{O}_{\tilde{Z}})=0$ and $H^1(\tilde{Z}, \mathcal{O}^*_{\tilde{Z}})$ is a finitely generated $\mathbb{Z}$-module.
 We can also show that $\alpha'_2$ is injective 
 by \cite[Lemma 2.2, Claim]{MR1286924} since the problem is local around 
 the singularities of $Z$. By applying the snake lemma to the diagram (\ref{eq:diagramsnake}), we obtain the surjectivity of $\tilde{\alpha}_2$. 
 \end{proof}
 
Since $H^1(\Delta, \mathcal{O}_{\Delta})=0$, the homomorphism $\alpha_3$ is injective. 
Thus, by the snake lemma, we see that $\alpha_1$ is surjective. 

Let $\mathcal{K}':= \Ker (\mathcal{K}_{(Z_n, \Delta_n)} \rightarrow \mathcal{K}_{(Z_{n-1}, \Delta_{n-1})})$. 
We see that $\mathcal{K}' \simeq \mathcal{O}_Z(- \Delta)$ since we have a commutative diagram 
\[
\xymatrix{ 
& 0 \ar[d] & 0 \ar[d] & 0 \ar[d] & \\
0 \ar[r] & \mathcal{K}' \ar[r] \ar[d] & 
\mathcal{K}_{(Z_n, \Delta_n)} \ar[r] \ar[d] & 
\mathcal{K}_{(Z_{n-1}, \Delta_{n-1})} \ar[r] \ar[d] & 
0 \\ 
0 \ar[r] & \mathcal{O}_Z \ar[r] \ar[d] & 
\mathcal{O}_{Z_n}^* \ar[r] \ar[d] & 
\mathcal{O}_{Z_{n-1}}^* \ar[r] \ar[d] & 
0 \\ 
0 \ar[r] & \mathcal{O}_{\Delta} \ar[r] \ar[d] & 
\mathcal{O}_{\Delta_n}^* \ar[r] \ar[d] & 
\mathcal{O}_{\Delta_{n-1}}^* \ar[r] \ar[d] & 
0.  \\  
 & 0 & 0 & 0 &  
}
\] 
By this and $H^2(Z, \mathcal{O}_Z(- \Delta)) =0$, we see that the homomorphism 
\[
\beta_1 \colon H^1(Z_n, \mathcal{K}_{(Z_n, \Delta_n)}) \rightarrow H^1(Z, \mathcal{K}_{(Z, \Delta)})
\] 
is surjective. 
Since we have a commutative diagram 
\[
\xymatrix{
H^1(Z_n, \mathcal{K}_{(Z_n, \Delta_n)}) \otimes_{\mathbb{Z}} A_n \ar[r]^{\beta_1} \ar[d] & 
 H^1(Z, \mathcal{K}_{(Z, \Delta)}) \otimes_{\mathbb{Z}} \mathbb{C} \ar[d]^{\alpha_1} \\
H^1(Z_n, \Omega^1_{Z_n/A_n}(\log \Delta_n)(-\Delta_n)) \ar[r]^{\beta_2} & 
H^1(Z, \Omega^1_Z( \log \Delta)(- \Delta)),   
}
\]
we see that the homomorphism $\beta_2$ is surjective. 
We see that $\beta_2 \otimes_{A_n} \mathbb{C}$ is an isomorphism by the base change theorem (\cite[Theorem 12.11 (a)]{Hartshorne}). 
Thus, by Nakayama's lemma, 
we see that $H^1(Z_n, \mathcal{K}_{(Z_n, \Delta_n)}) \otimes_{\mathbb{Z}} A_n  \rightarrow H^1(Z_n, \Omega^1_{Z_n/A_n}(\log \Delta_n)(-\Delta_n))$ 
is surjective. By this and the commutative diagram 
\[
\xymatrix{
H^1(Z_n, \mathcal{K}_{(Z_n, \Delta_n)}) \otimes_{\mathbb{Z}} A_n \ar[r]^{\beta'_1 \ \ \ \ } \ar[d] & 
 H^1(Z_{n-1}, \mathcal{K}_{(Z_{n-1}, \Delta_{n-1})}) \otimes_{\mathbb{Z}} A_{n-1} \ar[d] \\
H^1(Z_n, \Omega^1_{Z_n/A_n}(\log \Delta_n)(-\Delta_n)) \ar[r]^{\Phi_n \ \ \ \ } & 
H^1(Z_{n-1}, \Omega^1_{Z_{n-1}}( \log \Delta_{n-1})(- \Delta_{n-1})),   
}
\]
we see that $\Phi_n$ is surjective. 
Hence we obtain (i). 
We finish the proof of Theorem \ref{thm:pairunobs}.  
	\end{proof}

\begin{cor}\label{cor: wqfano3unobs}
	Let $X$ be a weak $\mathbb{Q}$-Fano $3$-fold. Then its deformation functor $\Def_X$ is unobstructed. 
	\end{cor}

\begin{proof} 
	Let $D \in |{-}m K_X|$ be a smooth divisor as in Theorem \ref{thm:pairunobs}. 
	Let $F: \Def_{(X,D)} \rightarrow \Def_X$ be the forgetful morphism. 
	Note that $H^1(D, \mathcal{N}_{D/X}) =0$ since we have an exact sequence 
	\[
	 H^1(X, \mathcal{O}_X(D)) \rightarrow H^1(D, \mathcal{N}_{D/X}) \rightarrow H^2(X, \mathcal{O}_X) 
	\]
	and both side terms are zero by the Kawamata-Viehweg vanishing theorem. Thus, by Proposition \ref{prop:forgetfulsmooth} and Theorem \ref{thm:pairunobs}, we see that $F$ is a smooth morphism and 
	this implies that $\Def_X$ is unobstructed 
	since $\Def_{(X,D)}$ is unobstructed. 
	\end{proof}

\section{$\mathbb{Q}$-smoothings of a weak $\mathbb{Q}$-Fano $3$-fold}

\subsection{Du Bois invariants of singularities}
We shall introduce two types of Du Bois invariants of a $3$-fold terminal singularity. 

\vspace{2mm}

First, we recall the invariant used in \cite[Section 1]{MR1358982} for the Gorenstein case and \cite{MR3692020} in the general case. 
Let $(U,p)$ be a germ of a $3$-fold terminal singularity and 
\[
\pi \colon V:= \Spec \oplus_{i=0}^{r-1} \mathcal{O}_U (i K_U) \rightarrow U
\]
 be its index one cover with the $\mathbb{Z}_r$-action with a point $q:= \pi^{-1}(p)$. 
Let $\nu \colon \tilde{V} \rightarrow V$ be a $\mathbb{Z}_r$-equivariant log resolution such that its exceptional divisor
 $F \subset \tilde{V}$ has a SNC support and $\tilde{V} \setminus F \simeq V \setminus \{q \}$.  
Let $V':= V \setminus \{q \}$ and 
\[
\tau_V \colon H^1(V', \Omega^2_{V'}(-K_{V'})) \rightarrow H^2_{F}(\tilde{V}, \Omega^2_{\tilde{V}}(-\nu^*K_V))
\] be the coboundary map of the local cohomology. 
Let $\tilde{\pi} \colon \tilde{V} \rightarrow \tilde{U}:= \tilde{V}/\mathbb{Z}_r$ be the finite morphism induced by $\pi$ and $E \subset \tilde{U}$  
the exceptional divisor of the birational morphism $\mu \colon \tilde{U} \rightarrow U$ induced by $\nu$. 
Let $U':= U \setminus \{ p \}$ and  $\mathcal{F}_U^{(0)}$ the $\mathbb{Z}_r$-invariant part of $\tilde{\pi}_* \Omega^2_{\tilde{V}}( -\nu^* K_V)$. 
Then we have the coboundary map 
\begin{equation}\label{phiUdescription}
\phi_U \colon H^1(U', \Omega^2_{U'}(-K_{U'})) \rightarrow H^2_{E}(\tilde{U}, \mathcal{F}_U^{(0)}) 
\end{equation}
as the $\mathbb{Z}_r$-invariant part of $\tau_V$. 
We have the following result. 

\begin{fact}\label{fact:phiU0}(\cite[Corollary 2.7]{MR3692020}) Let $(U,p)$ be a germ of a $3$-fold terminal singularity. Then the following are equivalent. 
\begin{enumerate}
\item[(i)] $\phi_U =0$. 
\item[(ii)] $(U, p)$ is a quotient singularity or an $A_{1,2}/4$-singularity. 
 \end{enumerate}
 Here, an $A_{1,2}/4$-singularity is the germ $(x^2+y^2 +z^3 +u^2=0)/ \mathbb{Z}_4(1,3,2,1)$. 
\end{fact}

We also use another invariant of a terminal singularity. 

\begin{defn}\label{defn:DBinvterminalb11}
	Let $\tilde{\pi} \colon \tilde{V} \rightarrow \tilde{U}$ and $F \subset \tilde{V}$ 
	be as above.  Let 
	\[
	\tilde{\pi}_* \Omega^1_{\tilde{V}}(\log F)(-F) = \bigoplus_{i=0}^{r-1} \mathcal{G}_U^{(i)}
	\]
	be the eigen-decomposition with respect to the $\mathbb{Z}_r$-action of $\pi$ such that 
	\[
	\mathcal{G}_{U}^{(i)}:= \{ s \in \tilde{\pi}_* \Omega^1_{\tilde{V}}(\log F)(-F) \mid 
	g \cdot s = \zeta_r^i s \}
	\]
	for $\zeta_r:= \exp(2 \pi \sqrt{-1}/r)$ and the generator $g:= \bar{1} \in \mathbb{Z}_r$. 
	Let 
	\[
	\mu^{(i)}(U,p):= \dim_{\mathbb{C}} H^0(\tilde{U}, \mathcal{G}_U^{(i)}). 
	\]
	From now, for a positive integer $m$ and 
	a $\mathbb{C}$-vector space $V$ (or a sheaf) with a $\mathbb{Z}_m$-action, 
	let 
	\[
	V^{(i)}:= \{ v \in V \mid g \cdot v = \zeta_m^i v \}
	\]
	be the eigenspace with eigenvalue $\zeta_r^i$ for $i=0, \ldots, m-1$.  
	\end{defn}
	
	\begin{rem}
	When $(U,p)$ is a $3$-fold Gorenstein terminal singularity, the invariant $\mu(U,p):=\mu^{(0)}(U,p)$ 
	is used to study smoothability of a Calabi-Yau $3$-fold with terminal singularities (\cite{MR1358982}). It is proved that $\mu(U,p) =0$ only when $(U,p)$ is either smooth 
	or an ODP (\cite[Theorem 2.2]{MR1358982}). 
\end{rem}

We use the invariant $\mu^{(1)}(U, p)$ to study $\mathbb{Q}$-smoothing of a weak $\mathbb{Q}$-Fano $3$-fold. 
The following can be useful to compute the invariant $\mu^{(1)}$. 

\begin{lem}\label{lem:mu-1ineq}
Let $(U,p)$ be a $3$-fold terminal singularity such that $\mu^{(1)}(U,p) =0$. 
Then we have 
\[
\mu^{(-1)}(U,p) \ge \dim T^1_{(U,p)} - \sigma^{(1)}(U,p), 
\]
where $\sigma^{(1)}(U,p)$ is the rank of the eigenspace 
$\Cl (V,q)^{(1)}$ of the local divisor class group $\Cl (V,q)$ of the germ $(V,q)$.  
\end{lem} 

\begin{proof}
We have an exact sequence of the $\mathbb{Z}_r$-eigenspaces 
\[
H^1(\tilde{V}, \Omega^1_{\tilde{V}}(\log F)(-F))^{(1)} 
\rightarrow H^1(\tilde{V}, \Omega^1_{\tilde{V}}(\log F))^{(1)} 
\rightarrow H^1(F, \Omega^1_{\tilde{V}}(\log F)|_F)^{(1)}  
\]
The assumption $\mu^{(1)}(U,p) =0$ implies that the L.H.S. is zero. 
Hence we obtain 
\[
h^1(\tilde{V}, \Omega^1_{\tilde{V}}(\log F))^{(1)} \le \sigma^{(1)}(U,p)
\]
since we have 
$
H^1(F, \Omega^1_{\tilde{V}}(\log F)|_F)^{(1)} \simeq \Gr_F^1 H^3_{\{ q\}}(V, \mathbb{C})^{(1)}$  
and $H^3_{\{ q\}}(V, \mathbb{C}) \simeq \Cl (V,q)$. 
By the local duality, we obtain 
\[
h^1(\tilde{V}, \Omega^1_{\tilde{V}}(\log F))^{(1)} = h^2_F(\tilde{V}, \Omega^2_{\tilde{V}}(\log F)(-F))^{(-1)}. 
\]
We also have an exact sequence of the eigenspaces 
\[
H^1(\tilde{V}, \Omega^2_{\tilde{V}}(\log F)(-F))^{(-1)} 
\rightarrow 
H^1(V', \Omega^2_{V'})^{(-1)} 
\rightarrow H^2_F(\tilde{V}, \Omega^2_{\tilde{V}}(\log F)(-F))^{(-1)}. 
\] 
Note that $H^1(V', \Omega^2_{V'})^{(-1)} \simeq H^1(U', \Omega^2_{U'}(-K_{U'})) \simeq T^1_{(U,p)}$. 
Also note that we have a surjection  
$d \colon H^1(\tilde{V}, \Omega^1_{\tilde{V}}(\log F)(-F))^{(-1)} \rightarrow H^1(\tilde{V}, \Omega^2_{\tilde{V}}(\log F)(-F))^{(-1)} $ which is an eigenpart of the surjection 
$d \colon H^1(\tilde{V}, \Omega^1_{\tilde{V}}(\log F)(-F)) \rightarrow H^1(\tilde{V}, \Omega^2_{\tilde{V}}(\log F)(-F))$ (See the proof of \cite[Theorem 1.1]{MR1358982}). 
By this exact sequence and the above inequality, we obtain 
\[
\dim T^1_{(U,p)} \le \mu^{(-1)}(U,p) + h^2_F(\tilde{V}, \Omega^2_{\tilde{V}}(\log F)(-F))^{(-1)} 
\le \mu^{(-1)}(U,p) + \sigma^{(1)}(U,p). 
\]
\end{proof}

\begin{eg}
Let us consider the terminal singularity $U := (xy+ f(z^r, u)=0)/ \mathbb{Z}_r(1,-1, a,0)$ 
for coprime positive integers $r, a$. 

For example, let $U:= (xy +z^2 +u^{2k+1}=0)/ \mathbb{Z}_2(1,1,1,0)$ 
for $k \in \mathbb{Z}_{>0}$. 
If $\mu^{(1)}(U,p)=0$, then, by Lemma \ref{lem:mu-1ineq} and the Gorenstein index $r=2$, 
we obtain  $\mu^{(1)}(U,p)= \mu^{(-1)}(U,p) \ge \dim (T^1_{(V,q)})^{(-1)} =k$ 
 since $(V,q)$ is locally factorial. 
This is a contradiction and we obtain $\mu^{(1)}(U,p) >0$. 

Similarly, we can compute $\mu^{(1)}(U,p) >0$ for $(U,p)= (xy +z^2 +u^{2k}=0)/ \mathbb{Z}_2(1,1,1,0)$ for $k \in \mathbb{Z}_{>1}$ as follows. 
Suppose that $\mu^{(1)}(U, p) =0$. 
Then we have $\sigma(V,q) =1$ and $\mu^{(1)}(U,p) = \mu^{(-1)}(U,p) \ge \dim T^1_{(U,p)} -1$ by Lemma \ref{lem:mu-1ineq} again. 
This is also a contradiction and we obtain $\mu^{(1)}(U,p) >0$.


For other terminal singularities, the computation may be possible, but complicated. 
\end{eg}

\subsection{On $\mathbb{Q}$-smoothability of a weak $\mathbb{Q}$-Fano $3$-fold}

By using the invariants of a terminal singularity as in the previous subsection, we can prove the following result on the existence of a global deformation of a weak $\mathbb{Q}$-Fano $3$-fold. 

\begin{thm}\label{thm:QsmQfactIndex1}
Let $X$ be a weak $\mathbb{Q}$-Fano $3$-fold. 

\begin{enumerate}
	\item[(i)] Then there exists a deformation 
$\phi \colon \mathcal{X} \rightarrow \Delta^1$ over an unit disk whose 
general fiber $\mathcal{X}_t$ satisfies the following: For $p \in \Sing \mathcal{X}_t$ and its 
Stein neighborhood $U_p$, 
we have $\mu^{(1)}(U_p, p) =0$. 
	\item[(ii)]
Let $D \in |{-}m K_X|$ be a smooth divisor for some positive integer $m >0$ such that $D \cap \Sing X = \emptyset$ which exists 
by the base point free theorem.  
Let  \[
Y:= \Spec \oplus_{i=0}^{m-1} \mathcal{O}_X(i K_X) \rightarrow X  
\]
be the cyclic cover determined by $D$. Assume that $Y$ is $\mathbb{Q}$-factorial. ($Y$ is 
the ``global index one cover'' in Theorem \ref{thm:wqfanoQsmintro}.)

Then $X$ can be deformed to a weak $\mathbb{Q}$-Fano $3$-fold with only quotient singularities and $A_{1,2}/4$-singularities. 
\end{enumerate}
\end{thm}

\begin{proof}
Let $\pi \colon Y \rightarrow X$ be the $\mathbb{Z}_m$-cover as in (ii) and 
let $\Sing Y =: \{p_1, \ldots, p_l \}$.  	
	
There exists a $\mathbb{Z}_m$-equivariant log resolution (\cite{MR1453072}) $\nu \colon  \tilde{Y} \rightarrow Y$ which induces an isomorphism 
$\nu^{-1}(Y \setminus \pi^{-1} \{p_1,\ldots, p_l \}) \rightarrow Y \setminus \pi^{-1} \{p_1,\ldots, p_l \}$ 
and $\nu^{-1}(\Sing Y)$ is a SNC divisor. Let   
$\mu \colon \tilde{X} := \tilde{Y}/ \mathbb{Z}_m \rightarrow X$ be 
a birational morphism induced by $\nu$. 

Let $U_i$ be a Stein neighborhood of $p_i$ for $i=1, \ldots, l$. 
Let $\pi_i \colon V_{i} := \pi^{-1}(U_i) \rightarrow U_i$ and 
$\nu_i \colon \tilde{V}_{i}:= \nu^{-1}(V_i) \rightarrow V_{i}$ be the morphisms induced by $\pi$ and $\nu$ 
respectively. 
Let $\tilde{U}_i := \tilde{V}_{i}/ \mathbb{Z}_m$ with a birational morphism $\mu_i \colon \tilde{U}_i 
\rightarrow U_i$ induced by $\nu_i$. 
 Let $F:= \Exc(\nu), E:= \Exc (\mu),  \Delta := \pi^{-1}(D) $ and $ L := \mathcal{O}_{Y}(\Delta) = \mathcal{O}_Y(\pi^*(-K_X))$.  
Also let $F_{i} := \Exc (\nu_{i})$ and 
$E_i := \Exc (\mu_i)$.  
 Let 
 \[
 \mathcal{F}^{(0)}:=  (\tilde{\pi}_* (\Omega^2_{\tilde{Y}} \otimes \nu^* L))^{\mathbb{Z}_m}
 \] be the $\mathbb{Z}_m$-invariant part and $\mathcal{F}_i^{(0)}:= \mathcal{F}^{(0)}|_{\tilde{U}_i}$ 
 be its restriction. 

Let $X':= X \setminus \{p_1, \ldots , p_l \}$. We have the following commutative diagram;  
\begin{equation}\label{minadiag}
\xymatrix{
      H^1(X', \Omega^2_{X'} \otimes \omega_{X'}^{-1} ) \ar[r]^{\oplus \psi_i} \ar[d]^{\oplus p_{U_i}} & 
\oplus_{i=1}^{l} H^2_{E_i}(\tilde{X}, \mathcal{F}^{(0)} )
\ar[d]_{\oplus \varphi_i}^{\simeq} \ar[r]^{\ \ \oplus \Psi_i} &  H^2 (\tilde{X}, \mathcal{F}^{(0)} ) \\
     \oplus_{i=1}^{l} H^1(U'_i, \Omega^2_{U'_i}\otimes \omega_{U'_i}^{-1}) \ar[r]^{\oplus \phi_{i}}& 
\oplus_{i=1}^{l}  H^2_{E_i}(\tilde{U}_i, \mathcal{F}_{i}^{(0)}). 
     }
 \end{equation}
Now let $V'_i:= V_i \setminus \pi_i^{-1}(p_i)$ and $\tilde{\Delta}:= \nu^{-1}(\Delta)$. 

\vspace{2mm}

\noindent(i) 
	Let $p_1, \ldots, p_{l'} \in \Sing X$ be the singular points such that 
	$\mu^{(1)}(U_i ,p_i) \neq 0$ for $i = 1, \ldots, l'$. 
	
	\begin{claim}\label{claim:notinjective}
		The homomorphism $\Psi_i \circ \varphi_i^{-1}$ is not injective for $i=1, \ldots, l'$. 
		\end{claim} 
		
		\begin{proof}[Proof of Claim]
The homomorphism is the $\mathbb{Z}_m$-eigenpart  
\[
H^2_{F_i}(\tilde{V}_i, \Omega^2_{\tilde{V}_i})^{(-1)} \xrightarrow{\simeq}
 H^2_{F_i}(\tilde{Y}, \Omega^2_{\tilde{Y}}(\tilde{\Delta}))^{(-1)} 
 \rightarrow H^2(\tilde{Y}, \Omega^2_{\tilde{Y}}(\tilde{\Delta}))^{(-1)} 
\]
The homomorphism is dual to the eigenpart  
\[
\gamma_i \colon H^1(\tilde{Y}, \Omega^1_{\tilde{Y}}(- \tilde{\Delta}))^{(1)} 
\rightarrow H^1(\tilde{V}_i, \Omega^1_{\tilde{V}_i})^{(1)} 
\]
of the restriction homomorphism. 
This homomorphism fits in the following commutative diagram
\[
\xymatrix{
H^1(\tilde{Y}, \Omega^1_{\tilde{Y}}(- \tilde{\Delta}))^{(1)} \ar[d] \ar[dr]^{\gamma_i} & \\
H^1(\tilde{Y}, \Omega^1_{\tilde{Y}})^{(1)} \ar[r] & H^1(\tilde{V}_i, \Omega^1_{\tilde{V}_i})^{(1)} \\
H^1(\tilde{Y}, \mathcal{O}_{\tilde{Y}}^*)^{(1)}_{\mathbb{C}} \ar[r] \ar[u]^{\simeq} & 
H^1(\tilde{V}_i, \mathcal{O}_{\tilde{V}_i}^*)^{(1)}_{\mathbb{C}}  \ar[u]^{\partial_{V_i}^{(1)}}
}
\]
Hence it is enough to check the corresponding eigenpart  
\[
\partial_{V_i}^{(1)} \colon  H^1(\tilde{V}_i, \mathcal{O}_{\tilde{V}_i}^*)^{(1)}_{\mathbb{C}}  
\rightarrow H^1(\tilde{V}_i, \Omega^1_{\tilde{V}_i})^{(1)} 
\]
of $\partial_{V_i} \colon H^1(\tilde{V}_i, \mathcal{O}_{\tilde{V}_i}^*)_{\mathbb{C}}  
\rightarrow H^1(\tilde{V}_i, \Omega^1_{\tilde{V}_i})$ is not surjective. 
We see that the dimension of its cokernel is 
$\mu^{(1)}(U_i, p_i)= h^1(\tilde{U}_i, \mathcal{G}_{U_i}^{(1)})$ by \cite[Proposition 2.1]{MR1358982}. 
Since this is non-zero by the assumption, we obtain the claim.  			
			\end{proof} 

Let 
\[
(\nu_i)_* \colon H^1(\tilde{V}_i, \Omega^2_{\tilde{V}_i}(-K_{\tilde{V}_i})) 
\rightarrow H^1(V'_i, \Omega^2_{V'_i}(-K_{V'_i}))
\] be the restriction homomorphism which 
can also be regarded as the blow-down morphism (cf. \cite[Section 2]{MR3692020}). 
Note that 
\begin{equation}\label{eq:nu_irelation}
\Image (\nu_i)_* \subset \Ker \tau_{V_i}
\end{equation}
as in \cite[(2.12)]{MR3692020}, where $\tau_{V_i} 
\colon H^1(V'_i, \Omega^2_{V'_i}(-K_{V'_i})) \rightarrow H^2_{F_i}(\tilde{V}_i, \Omega^2_{\tilde{V}_i}(\nu_i^*(-K_{V_i})))$ is 
the coboundary map. 
By using this relation, we can find a good deformation as follows. 

By Claim \ref{claim:notinjective}, there exists 
$\alpha_i \in H^2_{E_i}(\tilde{U}_i, \mathcal{F}_{i}^{(0)})$ such that 
$\Psi_i (\varphi_i^{-1}(\alpha_i)) = 0$. 
Hence there exists $\eta \in H^1(X', \Omega^2_{X'}(- K_{X'}))$ such that
 $\psi_i (\eta)= \varphi_i^{-1}(\alpha_i)$. 
Then we have $p_{U_i}(\eta) = \alpha_i \neq 0$, thus 
$p_{U_i}(\eta) \not\in \Ker \phi_i$. 
By this and the relation (\ref{eq:nu_irelation}), we see that $p_{U_i}(\eta) \not\in \Image (\nu_i)_*$. 
By arguing as in the proof of \cite[Theorem 1.4]{MR3692020}, we can deform singularity $p_i \in U_i$ as long as the invariant 
$\mu^{(1)}$ is non-zero. Thus, by Theorem \ref{thm:wqfano3intro}, we obtain a deformation with the required property. 

\vspace{2mm}

\noindent(ii)	The framework of the proof is similar to that of \cite[Proof of Theorem 1.4]{MR3692020}. 
	
	Let  $p_1,\ldots, p_{l'}$ for some $l' \le l$ be the singularities which are neither quotient singularities nor 
	 $A_{1,2}/4$-singularities. 

\begin{claim}\label{claim:compositionzero}
	We have $\Psi_i \circ (\varphi_i^{-1}) \circ \phi_i =0$ for all $i$. 
	\end{claim}
	
	\begin{proof}[Proof of Claim]

We see that the composition $\Psi_i \circ (\varphi_i^{-1}) \circ \phi_i$ is 
one of the $\mathbb{Z}_m$-eigenparts of the composition 
\begin{equation}
H^1(V'_i, \Omega^2_{V'_i}) \xrightarrow{\tau_{V_i}} H^2_{F_i}(\tilde{V}_i, \Omega^2_{\tilde{V}_i})
\xrightarrow{\simeq} H^2_{F}(\tilde{Y}, \Omega^2_{\tilde{Y}}(\tilde{\Delta}))
\rightarrow H^2(\tilde{Y}, \Omega^2_{\tilde{Y}}( \tilde{\Delta})). 
\end{equation}
Note that $\tilde{\Delta} \cap \tilde{V}_i = \emptyset$ for all $i$. 
This is dual to the restriction homomorphism 
\[
H^1(\tilde{Y}, \Omega^1_{\tilde{Y}}(- \tilde{\Delta})) \rightarrow H^1(V'_i, \Omega^1_{V'_i}) 
\]
and it can be decomposed as 
\[
H^1(\tilde{Y}, \Omega^1_{\tilde{Y}}(- \tilde{\Delta})) \rightarrow 
H^1(\tilde{Y}, \Omega^1_{\tilde{Y}}) \xrightarrow{\Phi_i} 
 H^1(V'_i, \Omega^1_{V'_i}) 
\]
and $\Phi_i$ fits in the following commutative diagram
\[
\xymatrix{
H^1(\tilde{Y}, \Omega^1_{\tilde{Y}}) \ar[r]^{\Phi_i} &  
 H^1(V'_i, \Omega^1_{V'_i}) \\
H^1(\tilde{Y}, \mathcal{O}_{\tilde{Y}}^*) \otimes_{\mathbb{Z}} \mathbb{C} 
\ar[r]^{\Phi'_i} \ar[u]^{\simeq} & H^1(V'_i, \mathcal{O}_{V'_i}^*) \otimes_{\mathbb{Z}} \mathbb{C}.  
\ar[u]
}
\]
We see that $\Phi'_i$ is zero since $Y$ is $\mathbb{Q}$-factorial by \cite[12.1.6]{MR1149195} (cf. the 
proof of \cite[Proposition 1.2]{MR1358982}). Thus we obtain the claim. 
\end{proof}

By Fact \ref{fact:phiU0}, for $i=1, \ldots, l'$, there exists an element $\eta_i \in H^1(U'_i, \Omega^2_{U'_i}(-K_{U'_i}))$ such that $\phi_i (\eta_i) \neq 0$. 
By Claim \ref{claim:compositionzero}, there exists 
$\eta \in H^1(X', \Omega^2_{X'}(-K_{X'}))$ such that $\psi_i(\eta)= \varphi_i^{-1}(\phi_i(\eta_i))$. 

By the relation (\ref{eq:nu_irelation}) and $p_{U_i}(\eta) - \eta_i \in \Ker \phi_i$, 
we see that $p_{U_i}(\eta) \not\in \Image (\nu_i)_*$, where we use the inclusion 
$H^1(U'_i, \Omega^2_{U'_i}(-K_{U'_i})) \subset H^1(V'_i, \Omega^2_{V'_i}(-K_{U'_i}))$ 
and regard $\phi_i$ as the restriction of $\tau_{V_i}$. 
By arguing as in the proof of \cite[Theorem 1.4]{MR3692020}, we can deform singularity $p_i \in U_i$ as long as 
$\phi_i \neq 0$. 
By Fact \ref{fact:phiU0} and Theorem \ref{thm:wqfano3intro}, we obtain a required deformation. 
\end{proof}

\section{Deformations of a $\mathbb{Q}$-Fano $3$-fold via its canonical covering stack}\label{section:stack}
In this section, we explain the canonical covering stack associated to a $3$-fold 
with only terminal singularities. 
We use it to prove the unobstructedness of deformations of a $\mathbb{Q}$-Fano $3$-fold. 

\subsection{Preliminaries on canonical covering stacks}
Let $X$ be a $3$-fold with only terminal singularities. 
Let $\Sing X=: \{p_1, \ldots, p_l \}$, $p_i \in U_i$ a small affine neighborhood of $p_i$ 
such that $U_i \cap \Sing X = \{p_i \}$,     
and $\pi_i \colon V_i \to U_i$ be the index one cover for $i=1, \ldots, l$. 
Let $V_0:= X \setminus \Sing X$ and $\pi_0 \colon V_0 \to X$ the open immersion. 
Let $V:= \coprod_{i=0}^l V_i$ and $\pi \colon V \to X$ the morphism such that 
$\pi|_{V_i} = \pi_i$ for $i=0, \ldots, l$. 
Let $G_i:= \Gal (V_i/U_i) \simeq \mathbb{Z}_{r_i}$ and 
\[
W:= (\coprod_{i} V_i \times G_i )\sqcup (\coprod_{i \neq j} V_i \times_{X} V_j ). 
\] 
We shall define an \'{e}tale groupoid space 
\[
\mathfrak{G}(V) := (\xymatrix{ 
W{}_s \underset{V}{\times} {}_t W \ar[r]^{\quad\;\;\; m} &
W \ar@(ur,ul)[]_i \ar@<2ex>[r]^s \ar[r]^t &
V \ar@<1ex>[l]^e
})
\]
 by setting $m, s, t$ as follows: For $i = 1, \ldots, l$ and $i \neq j$,  let 
 \[
 W_i:= V_i \times G_i, \ \ \ \ W_{ij}:= V_i \times_X V_j.
 \] 
For $i=1, \ldots, l$,  let $s_i \colon V_i \times G_i \rightarrow V_i$ be the projection 
and $t_i \colon V_i \times G_i \rightarrow V_i$ be the morphism defining the $G_i$-action on $V_i$, that is, the diagram  
\[
\xymatrix{
V_i \times G_i \ar@<-.5ex>[r]_{t_i} \ar@<.5ex>[r]^{s_i} & V_i 
}
\]
is a groupoid space defining the quotient stack $[V_i / G_i]$.  
For $i \neq j$, let $s_{i, j} \colon V_i \times_X V_j \rightarrow V_i$ and  
$t_{i, j} \colon V_i \times_X V_j \rightarrow V_j$ 
be the projections. 
We define \[
s, t \colon W \rightrightarrows V
\] by $s|_{V_i \times G_i}= s_i$ (resp. $t|_{V_i \times G_i}= t_i$) and $s|_{V_i \times_X V_j} = s_{i, j}$ 
(resp. $t|_{V_i \times_X V_j} = t_{i, j}$). 
In order to define the multiplication map 
$m \colon W \times_V W \rightarrow W$, we use the following morphisms; 
\[
m_{i, i} \colon (V_i \times G_i) \times_{V_i} (V_i \times G_i) \rightarrow V_i \times G_i, 
\]  
\[
m_{i, (j, i)} \colon (V_i \times G_i) \times_{V_i} (V_j \times_X V_i ) \rightarrow V_j \times_X V_i,   
\]
\[
m_{(i,j), i} \colon  (V_i \times_X V_j )  \times_{V_i} (V_i \times G_i) \rightarrow V_i \times_X V_j, 
\]
\[
m_{(j,k), (i, j)}  (V_j \times_X V_k )  \times_{V_j} (V_i \times_X V_j) \rightarrow V_i \times_X V_k, 
\]
where $m_{i, i}$ is induced by the multiplication of the group $G_i$, $m_{(i,j), i} $ and $m_{(i,j), i}$ 
are the base change of the $G_i$-action morphism $t_i$, and $m_{(j,k), (i, j)}$ is the natural projection of the fiber products. 
We define $m$ by \[
m|_{W_i \times W_i} = m_{i, i}, \ \ \ \  m|_{W_i \times W_{ji}} = m_{i, (j,i)}, 
\]
\[
m|_{W_{ij} \times_{V_i} W_i} = m_{(i,j), i}, \ \ \ \ m|_{W_{jk} \times W_{ij}} = m_{(j,k), (i,j)}. 
\]

Let $\mathfrak{X}$ be the Deligne-Mumford stack associated to the groupoid space $W \rightrightarrows V$ (cf.\ \cite[Definition 6.1]{Kawamata_DEKE}). 
Let $\gamma \colon \mathfrak{X} \rightarrow X$ be the morphism to the coarse moduli space.
Note that the coarse moduli space of $\mathfrak{X}$ is isomorphic to $ X$ since it can be constructed by gluing $V_i/G_i \simeq U_i$ 
with respect to the same gluing morphism as $X$. 
We can define a functor $\Def_{\mathfrak{X}} \colon \Art_{\mathbb{C}} \rightarrow (Sets)$ 
of deformations of the stack $\mathfrak{X}$ over Artinian rings as in the case of schemes.  

\begin{rem}
Aoki (\cite[Proposition 3.2.5]{MR2099768}) pointed out that $\Def_{\mathfrak{X}}$ is isomorphic to the deformation functor 
	$\Def_{(W \rightrightarrows V)}$ of the \'{e}tale groupoid space $(W \rightrightarrows V)$. 
	The deformation functor of an \'{e}tale groupoid space can be defined in an obvious way.  
\end{rem}

\begin{rem}
	Hacking (\cite[Section 3]{MR2078368}) considered a canonical covering stack of a slc surface and its deformation theory. 
	The theory is parallel in the case of a 3-fold with terminal singularities. 
	\end{rem}

\begin{lem}(cf. \cite[Proposition 3.7]{MR2078368})
There is an isomorphism of functors
\begin{equation}\label{eq:coarsefunctorisom}
c \colon \Def_{\mathfrak{X}} \xrightarrow{\sim} \Def_X 
\end{equation}
which sends a deformation of $\mathfrak{X}$ to its coarse moduli space. 
\end{lem}

\begin{proof}
	We can construct a natural transformation $c' \colon \Def_X \rightarrow \Def_{\mathfrak{X}}$ as follows; 
	Let $X_A \rightarrow \Spec A$ be a deformation of $X$ over $A \in \Art_{\mathbb{C}}$. 
	Let $\{ U_i \}_{i=1}^n$ be an affine open sets of $X$ as above and $U_{i,A} \rightarrow \Spec A$ be a deformation of $U_i$ induced by $X_A$. 
	Let $\iota_i \colon U'_i:= U_i \setminus \{p_i \} \rightarrow U_i$ be the open immersion from the smooth locus for $i=1, \ldots, l$ and 
	  $\omega_{U_{i,A}}^{[j]}:= (\iota_i)_* \omega^{\otimes j}_{U'_{i,A}}$ for $j \in \mathbb{Z}$. 
	  We see that $\omega_{U_{i,A}}^{[j]}$ is flat over $A$ by Proposition \ref{prop:S3flat}.   
	Thus we see that 
	\[
	V_{i,A}:= \Spec \oplus_{j=0}^{r_i -1} \omega_{U_{i,A}/A}^{[j]} \rightarrow \Spec A 
	\] 
 is a $G_i$-equivariant deformation of $V_i$.  
	Let $V_{0,A}$ be a deformation of $V_0$ induced by $X_A$, 
	 $V_A :=  \coprod_{i=0}^l V_{i,A}$ and 
	\[
	W_A:= (\coprod_{i} V_{i,A} \times G_i )\sqcup (\coprod_{i \neq j} V_{i,A} \times_{X_A} V_{j,A} ). 
	\] We can construct an \'{e}tale groupoid space 
	$W_A \rightrightarrows V_A$ similarly as $W \rightrightarrows V$ and this defines 
	an element of $ \Def_{(W \rightrightarrows V)}(A)$, thus 
	an element $c'(X_A):= \mathfrak{X}_A \in \Def_{\mathfrak{X}}(A)$. 
	We see that $\mathfrak{X}_A$ is flat over $A$ since we can check it locally. 
	
	We can check that $c'$ is an inverse of $c$ as follows:  
	
	Given $X_A \in \Def_X(A)$. 
	We see that the coarse moduli space of the groupoid space $W_A \rightrightarrows V_A$ is isomorphic to $X_A$ since 
	it can be constructed by gluing $V_{i,A}/G_i \simeq U_{i, A}$ and the gluing isomorphism is same as $X_A$. 
	
	Conversely, given $\mathfrak{X}_A \in \Def_{\mathfrak{X}}(A)$. 
	Then we have the corresponding \'{e}tale groupoid space 
	\[
	W'_A= \coprod_{i} W'_{i, A} \sqcup \coprod_{i \neq j} W'_{ij, A} \rightrightarrows V'_A = \coprod_{i=0}^l V'_{i,A}
	\] which is a deformation of 
	$W \rightrightarrows V$. We can check  
	that $W'_{i,A} \rightrightarrows V'_{i, A}$ 
	is isomorphic to a groupoid space 
	$V'_{i, A} \times G_i \rightrightarrows V'_{i, A}$. 
	Hence the coarse moduli space $c(\mathfrak{X}_A)$ of $\mathfrak{X}_A$ has an open covering by affine schemes isomorphic to 
	$U'_{i, A}:= V'_{i, A}/G_i$. 
	We obtain the same groupoid spaces from $\mathfrak{X}_A$ and $c'(c(\mathfrak{X}_A))$ by the construction of $c'$. 
	\end{proof}

We can construct obstructions for deformations of $\mathfrak{X}$ as follows.  

\begin{prop}
	Let $X$ be a $3$-fold with terminal singularities and 
	$\mathfrak{X}$ its canonical covering stack. 
	
	Then we can define an obstruction $o_{\xi_n} \in \Ext^2_{\mathcal{O}_{\mathfrak{X}}}( \Omega^1_{\mathfrak{X}}, \mathcal{O}_{\mathfrak{X}} )$ 
	to lift a deformation $\xi_n \in \Def_{\mathfrak{X}}(A_n)$ to $A_{n+1}$.  
	\end{prop}
	
	\begin{proof} 
		The construction is parallel to \cite[Proposition 2.4.8]{MR2247603} or \cite[Proposition 2.6]{MR3419958}.  Let us recall the construction. 
		
		Let $\phi_n \colon \mathfrak{X}_n \rightarrow \Spec A_n$ be 
		the given deformation. 
		By pulling back the exact sequence $0 \rightarrow (t^{n+1}) \xrightarrow{d} \Omega^1_{A_{n+1}/ \mathbb{C}} 
		\otimes_{A_{n+1}} A_n \rightarrow \Omega^1_{A_n/ \mathbb{C}} \rightarrow 0$ to $\mathfrak{X}_n$,  
		we obtain 
		\[
		0 \rightarrow \mathcal{O}_{\mathfrak{X}}
		\rightarrow \mathcal{O}_{\mathfrak{X}_n}
		\rightarrow \mathcal{O}_{\mathfrak{X}_{n-1}}
		\rightarrow 0. 
		\]
		We also have an exact sequence 
		\[
		0 \rightarrow \phi_n^* \Omega^1_{A_n/ \mathbb{C}} 
		\rightarrow \Omega^1_{\mathfrak{X}_n / \mathbb{C}} 
		\rightarrow \Omega^1_{\mathfrak{X}_n / A_n} 
		\rightarrow 0. 
		\] 
		By combining these sequences, we obtain an exact sequence 
		\[
		0 \rightarrow \mathcal{O}_{\mathfrak{X}}
		\rightarrow \mathcal{O}_{\mathfrak{X}_n} 
		\rightarrow \Omega^1_{\mathfrak{X}_n / \mathbb{C}} 
		\rightarrow \Omega^1_{\mathfrak{X}_n / A_n} 
		\rightarrow 0. 
		\] 
		This determines an element $o_{\xi_n} \in
		\Ext^2_{\mathcal{O}_{\mathfrak{X}_n}}( \Omega^1_{\mathfrak{X}_n/A_n}, \mathcal{O}_{\mathfrak{X}} ) \simeq \Ext^2_{\mathcal{O}_{\mathfrak{X}}}( \Omega^1_{\mathfrak{X}}, \mathcal{O}_{\mathfrak{X}} )$, where the isomorphism follows similarly as \cite[Lemma 2.5]{MR3419958}. 
		We can check that, if $o_{\xi_n} =0$, then the lift $\xi_{n+1} \in \Def_{\mathfrak{X}}(A_{n+1})$ of $\xi_n$ exists as \cite[Proposition 2.6]{MR3419958}. 
		\end{proof}

	\subsection{Proof of unobstructedness}
	
	\begin{thm}\label{thm:QFano3unobs}
		Let $X$ be a $\mathbb{Q}$-Fano $3$-fold. 
		Then the deformation functor $\Def_X$ is unobstructed. 
		\end{thm}
	
	\begin{proof}
		Let $\gamma \colon \mathfrak{X} \rightarrow X$ be the canonical covering stack of $X$ constructed as above. 
		By the isomorphism (\ref{eq:coarsefunctorisom}), it is enough to show that $\Def_{\mathfrak{X}}$ is a smooth functor. 
		We have isomorphisms
		\[
		\Ext^2_{\mathcal{O}_{\mathfrak{X}}}(\Omega^1_{\mathfrak{X}}, \mathcal{O}_{\mathfrak{X}}) \simeq 
		\Ext^2_{\mathcal{O}_{\mathfrak{X}}}(\Omega^1_{\mathfrak{X}} \otimes \omega_{\mathfrak{X}}, \omega_{\mathfrak{X}}) \simeq 
		H^1(\mathfrak{X}, \Omega^1_{\mathfrak{X}} \otimes \omega_{\mathfrak{X}})^{\vee}.   
		\]
		The first isomorphism follows since $\omega_{\mathfrak{X}}$ is invertible. This is the main advantage of considering the canonical covering stack.  
		The second isomorphism follows from the Serre duality on a Deligne-Mumford stack (\cite[Corollary 2.10]{Nironi:2008aa}). 
		\begin{claim}
		We have an isomorphism 
		\[
		H^1(\mathfrak{X}, \Omega^1_{\mathfrak{X}} \otimes \omega_{\mathfrak{X}}) \simeq H^1(X, \iota_* (\Omega^1_{X'} \otimes \omega_{X'})),  
		\]
		 where $\iota \colon X' \hookrightarrow X$ is an open immersion of the smooth part $X'$ of $X$. 
		 \end{claim}
		 \begin{proof}[Proof of Claim]
		 We can check this by the construction of $\gamma \colon \mathfrak{X} \rightarrow X$. 
		 Note that $\gamma_*  (\Omega^1_{\mathfrak{X}} \otimes \omega_{\mathfrak{X}}) $ is reflexive since  
		  we have 
		 \[
		 \gamma_*  (\Omega^1_{\mathfrak{X}} \otimes \omega_{\mathfrak{X}}) |_{U_i} \simeq \left (\pi_i)_* (\Omega^1_{V_i} \otimes \omega_{V_i})\right)^{G_i}
		 \] 
		 on an affine neighborhood $U_i$ of $p_i$ for $i=1, \ldots, l$ and $\Omega^1_{V_i}$ is $S_2$ sheaf on $V_i$ with only 
		 Gorenstein terminal singularities. 
		 Thus we have an isomorphism 
		 \[
		 \gamma_*  (\Omega^1_{\mathfrak{X}} \otimes \omega_{\mathfrak{X}}) \simeq \iota_* (\Omega^1_{X'} \otimes \omega_{X'}) 
		 \] 
		 since both sheaves are reflexive and coincide on $X'$. This isomorphism induces the isomorphism of cohomology groups. 
		 \end{proof}
		 
		 Thus it is enough to check $H^1(X, \iota_* (\Omega^1_{X'} \otimes \omega_{X'}))=0$. 
		 This can be checked by a variant of Lefschetz hyperplane section theorem as in \cite[Theorem 2.11]{MR3419958}. 
		 \end{proof}

\begin{rem}
	The proof of Theorem \ref{thm:QFano3unobs} gives a new proof of \cite[Theorem 1.7]{MR3419958}. 
	In the proof of \cite[Theorem 1.7]{MR3419958}, we need to compare delicately   
	the deformations of a $\mathbb{Q}$-Fano $3$-fold and its smooth part. 
	The new proof avoids this issue.
	\end{rem}

\vspace{5mm}

By using the canonical covering stack, 
we can also prove the unobstructedness of $\mathbb{Q}$-Gorenstein deformations 
of a log del Pezzo surface. 

\begin{thm}(cf. \cite[Proposition 3.1]{Hacking-Prokhorov}, \cite[Lemma 6]{Akhtar:2015aa})
	Let $S$ be a log weak del Pezzo surface, that is, a normal projective surface with only klt singularities 
	with a nef and big anticanonical divisor $-K_S$. 
	
	Then the $\mathbb{Q}$-Gorenstein deformation functor $\Def^{qG}_S$ is unobstructed. 
	(See \cite[Definition 2.2]{MR2488488} for the definition of $\Def^{qG}_S$) 
	\end{thm}

\begin{proof}
	Let $\mathfrak{S} \rightarrow S$ be the canonical covering stack of $S$ 
	induced from canonical coverings of singularities of $S$ (\cite[3.1]{MR2078368}). Since we have $\Def_{\mathfrak{S}} \simeq \Def^{qG}_S$ (\cite[Proposition 3.7]{MR2078368}), 
	it is enough to show that $\Def_{\mathfrak{S}}$ is unobstructed. 
	
	We see that \[
	\Ext^2_{\mathcal{O}_{\mathfrak{S}}}( \Omega^1_{\mathfrak{S}}, \mathcal{O}_{\mathfrak{S}}) \simeq \Ext^2_{\mathcal{O}_{\mathfrak{S}}}( \Omega^1_{\mathfrak{S}} \otimes \omega_{\mathfrak{S}}, 
	\omega_{\mathfrak{S}}) \simeq H^0(\mathfrak{S}, \Omega^1_{\mathfrak{S}} \otimes \omega_{\mathfrak{S}}).
	\] 
	We also see that 
	\[
	H^0(\mathfrak{S}, \Omega^1_{\mathfrak{S}} \otimes \omega_{\mathfrak{S}}) \simeq H^0(S, (\Omega^1_S \otimes \omega_S)^{\vee \vee}) 
	\simeq \Hom (\mathcal{O}_S(-K_S), (\Omega^1_S)^{\vee \vee}).  
	\]
	This vanishes by the Bogomolov-Sommese vanishing theorem on a log canonical variety. Thus we are done. 
	\end{proof}

\section*{Acknowledgment}
	The author would like to thank Professor Yoshinori Namikawa for useful comments. 
	Theorem \ref{thm:wqfano3intro} is an answer to his question. 
	The author would like to thank Professor Alessio Corti for his lecture in Udine in 2014 
	which leads the author to  
	 an idea of using canonical covering stacks. 
	The author is also grateful to the anonymous referee for pointing out many missing 
	parts.

\bibliographystyle{amsalpha}
\bibliography{sanobibs}

\end{document}